\documentclass{amsart}

\usepackage{amsfonts,amssymb,amsmath,enumerate,verbatim,mathtools,tikz,bm,mathrsfs,tikz-cd,hyperref,comment,stmaryrd,amsthm}
\usepackage{colonequals}
\usepackage{adjustbox}
\usepackage[all,pdf]{xy}
\hypersetup{
    colorlinks=true,
    linkcolor=blue,
    filecolor=magenta,      
    urlcolor=cyan,
}
\usepackage{cleveref}

\makeatletter
\@namedef{subjclassname@2020}{\textup{2020} Mathematics Subject Classification}
\makeatother

\author[Jian Liu]{Jian Liu}
\address{School of Mathematics and Statistics, and Hubei Key Laboratory of Mathematical Sciences,  Central China Normal University,  Wuhan 430079, P.R. China}
\email{jianliu@ccnu.edu.cn}


\DeclareMathOperator{\depth}{depth}

\DeclareMathOperator{\h}{H}

\newcommand{\Z}{\mathbb{Z}}

\newcommand{\D}{\mathsf{D}}

\pgfdeclarelayer{bg}    
\pgfsetlayers{bg,main}

\newcommand{\p}{\mathfrak{p}}

\DeclareMathOperator{\Tr}{Tr}
\DeclareMathOperator{\Coker}{coker}

\DeclareMathOperator{\pd}{pd}
\DeclareMathOperator{\Gpd}{Gpd}

\DeclareMathOperator{\add}{add}

\DeclareMathOperator{\id}{id}

\DeclareMathOperator{\TF}{TF}

\DeclareMathOperator{\Hom}{Hom}

\DeclareMathOperator{\Ext}{Ext}

\DeclareMathOperator{\Tor}{Tor}

\DeclareMathOperator{\G}{G}

\DeclareMathOperator{\RHom}{\mathsf{RHom}}

\newcommand{\mo}{\mathsf{mod}}

\newcommand{\proj}{\mathsf{proj}}

\newcommand{\Gproj}{\mathsf{Gproj}}

\newtheorem{theorem}{Theorem}[section]

\newtheorem{proposition}[theorem]{Proposition}

\newtheorem{lemma}[theorem]{Lemma}
\newtheorem{corollary}[theorem]{Corollary}

\theoremstyle{definition}
\newtheorem{example}[theorem]{Example}
\newtheorem{remark}[theorem]{Remark}

\newtheorem{definition}[theorem]{Definition}

\newtheorem{chunk}[theorem]{}

\usepackage[utf8]{inputenc}

\newtheorem*{ack}{Acknowledgements}

\title[(Gorenstein) transpose and $k$-torsionfree modules]{Base change of  (Gorenstein) transpose, $k$-torsionfree modules, and quasi-faithfully flat extensions}

\begin{document}
\maketitle
\begin{abstract}
Let $\varphi\colon R \rightarrow A$ be a finite ring homomorphism, where $R$ is a two-sided Noetherian ring, and let $M$ be a finitely generated left $A$-module. Under suitable homological conditions on $A$ over $R$,   we establish a close relationship between the classical transpose of $M$ over $A$ and the Gorenstein transpose of a certain syzygy module of $M$ over $R$. As an application, for each integer $k>0$, we provide a sufficient condition under which $M$ is $k$-torsionfree over $A$ if and only if a certain syzygy of $M$ over $R$ is $k$-torsionfree over $R$,  extending a result of Zhao.
We introduce the notion of quasi-faithfully flat extensions and show that, under suitable assumptions, the extension closedness of the category of $k$-torsionfree modules over $R$ is equivalent to that over $A$.  An application is an affirmative answer to a question posed by Zhao concerning quasi $k$-Gorensteiness, in the case where both $R$ and $A$ are Noetherian algebras.
Finally, when $\varphi$ is a separable split Frobenius extension, it is proved that the category of $k$-torsionfree $R$-modules has finite representation type if and only if the same holds over $A$, with applications to skew group rings.
\end{abstract}
\section{Introduction}
The study of Gorenstein homological algebra has become a central topic in modern representation theory and commutative algebra, offering powerful frameworks for investigating homological properties; see \cite{Auslander-Bridger, Christensen-book, EJ}. Among these tools, the concepts of transpose and $k$-torsionfree modules play a fundamental role, especially within the framework of Auslander–Reiten theory; see \cite{Auslander-Bridger, ARS, Iyama2007}. Meanwhile, ring extensions, such as Frobenius extensions (cf. \Cref{def of Frobenius}) and their generalizations, provide a natural framework to examine how these homological properties are affected by changes of rings.

 Buchweitz \cite[8.2]{Buchweitz} observed that for the integral group ring extension \( \mathbb{Z} \to \mathbb{Z}G \) of a finite group $G$, a finitely generated left \( \mathbb{Z}G \)-module is Gorenstein projective if and only if its underlying \( \mathbb{Z} \)-module is Gorenstein projective, equivalently, free as a \( \mathbb{Z} \)-module. The ring extension $\Z\rightarrow \Z G$ is a classical example of a Frobenius extension \cite{Kas}. Building on this, Chen \cite{Chen} introduced a generalization of Frobenius extensions and proved that finitely generated Gorenstein projective modules transfer well along such extensions. This motivated Ren \cite{Ren-SCM} and Zhao \cite{Zhao2019} to study arbitrary (not necessarily finitely generated) Gorenstein projective modules within the framework of Frobenius extensions. Furthermore, the author and Ren \cite{Liu-Ren} systematically studied the ascent and descent of Gorenstein homological properties, providing equivalent characterizations of such properties for a ring homomorphism $\varphi\colon R \to A$ under mild assumptions.

Recently, Zhao \cite{Zhao2024} established an interesting result: 
    \emph{for a Frobenius extension $\varphi\colon R \rightarrow A$, a finitely generated left $A$-module is $k$-torsionfree if and only if it is $k$-torsionfree as a left $R$-module.} 
    Recall that a finitely generated left $R$-module $M$ is $k$-torsionfree if  $\Ext^i_R(\Tr_R(M), R) = 0$ for all $1 \leq i \leq k$, where $\Tr_R(M)$ denotes the transpose of $M$ over $R$ (cf. \ref{def-torsionfree}).

    
    In Section \ref{Section 3}, we investigate how the transpose behaves along ring homomorphisms. Let $M$ be a finitely generated left $A$-module. Our first main result, \Cref{T1}, establishes a close connection between the classical transpose of $M$ over $A$, denoted by $\Tr_A(M)$, and the Gorenstein transpose (cf. \ref{Gorenstein transpose}) of the $n$-th syzygy of $M$ over $R$, denoted by $\Tr_R^G(\Omega^n_R(M))$, introduced by Huang and Huang \cite{Huang-Huang} under change of rings. We point out that it is quite different from the characterization of the Gorenstein transpose via the classical transpose in loc. cit., as well as from the result presented in \cite{Zhao-Sun}; see discussions in \Cref{compare}. 
Theorem \ref{T1} provides a direct approach to Zhao's result \cite{Zhao2024}, in contrast to Zhao’s original proof, which relies on the characterization of $k$-torsionfree modules via projective approximations, as established by Auslander and Bridger \cite{Auslander-Bridger}; see \Cref{C1}. 
\begin{theorem}[See \ref{essential}]\label{T1}
Let $\varphi\colon R \to A$ be a finite ring homomorphism, where $R$ is a Noetherian ring. Suppose that $A$, viewed as a left $R$-module, has finite Gorenstein projective dimension, and that there exists  $n \geq 0$ such that $\operatorname{Ext}^i_R(A, R) = 0$ for all $i \neq n$. Then for each finitely generated left $A$-module $M$ and the transpose $\Tr_A(M)$, there exists a short exact sequence
\[
0 \longrightarrow K \longrightarrow \operatorname{Tr}^G_R(\Omega^n_R(M)) \longrightarrow \operatorname{Tr}_A(M) \otimes_A \operatorname{Ext}^n_R(A, R) \longrightarrow 0
\]
 of right $R$-modules, where the projective dimension of $K$ is at most $n - 1$.
\end{theorem}
In the above result, a module with projective dimension at most $-1$ is, by convention, the zero module. 

As an application of Theorem \ref{T1}, we obtain Corollary \ref{C1}, which generalizes Zhao’s result \cite{Zhao2024} on $k$-torsionfreeness. Specifically, Zhao established the case $n = 0$ in \Cref{C1}(1) under a stronger assumption that $\varphi$ is a Frobenius extension.
 It is worth emphasizing that a Frobenius extension requires $A$ to be projective as both a left and right $R$-module, and $\Hom_R(A, R)$ to be projective as a left $A$-module. In contrast, our assumptions are significantly weaker and are satisfied by a broader class of ring homomorphism, such as Frobenius extensions, homomorphisms between Gorenstein local rings, complete intersection maps; see \Cref{Section 4} for further details.

\begin{corollary}[See \ref{AD}]\label{C1}
Let $\varphi\colon R \to A$ be a finite ring homomorphism, where $R$ is a Noetherian ring. Suppose that $A$, viewed as a left $R$-module, has finite Gorenstein projective dimension, and that there exists $n \geq 0$ such that
$\Ext^i_R(A, R) = 0$ for all $i \neq n$,and $\Ext^n_R(A, R)$ is projective as a left $A$-module.
Then for each finitely generated left $A$-module $M$ and $k > 0$, we have{\rm :}
\begin{enumerate}
    \item If $n = 0$, then $M$ is $k$-torsionfree over $A$ if and only if it is $k$-torsionfree over $R$.

    \item Suppose $n \geq 1$ and $R$ is commutative. If the $n$-syzygy $\Omega^n_R(M)$ of $M$ over $R$ is $(k+n)$-torsionfree over $R$, then $M$ is $k$-torsionfree over $A$. The converse holds if $R_\p$ is Gorenstein for all prime ideals $\p$ of $R$ with $\depth(R_\p) \leq n - 1$.
\end{enumerate}
\end{corollary}

We also investigate the (left) quasi $k$-Gorenstein rings introduced by Huang \cite{Huang1999}. As shown by Huang, for a Noetherian algebra, the quasi $k$-Gorensteinness can be characterized by the extension closedness of the category of $k$-torsionfree modules. Recently, Dey and Takahashi \cite{DT} observed that, for a commutative Noetherian ring $R$ with depth $t$, the Gorensteinness of $R$ can be detected by the extension closedness of the category of $(t+1)$-torsionfree modules.

In Section \ref{Section 5}, we introduce the notion of \emph{quasi-faithfully flat extensions} (cf. \Cref{def of quasi-faithfully flat extension}) and establish \Cref{T2} concerning the extension closedness of the category of $k$-torsionfree modules under change of rings. 
We also prove that if $R$ is a Frobenius algebra over a field $k$, then any $k$-algebra homomorphism $
\varphi\colon R \rightarrow A$ is a quasi-faithfully flat extension; see \Cref{example-weak-ff}.
When the ring homomorphism $\varphi$ is further assumed to be a Frobenius extension, the forward implication of \Cref{T2} was previously obtained by Zhao \cite{Zhao2024}. It is worth mentioning that the assumption of \Cref{T2} does not necessarily imply that $A$ is projective as an $R$-module; see \Cref{Example-not proj}. 
\begin{theorem}[See \ref{AD of extension closedness}]\label{T2}
 Let $\varphi\colon R \rightarrow A$ be a finite ring homomorphism such that $\varphi$ is quasi-faithfully flat, where $R$ is a Noetherian ring.  Suppose that $A$, viewed as a left $R$-module, is Gorenstein projective, and $\Hom_R(A,R)$ is projective as a left $A$-module. Then for each $k > 0$, the category of $k$-torsionfree left $R$-modules is extension closed if and only if so is the category of $k$-torsionfree left $A$-modules.
\end{theorem}    
 Combining with Huang's characterization \cite{Huang1999} of the quasi $k$-Gorensteinness,  \Cref{C2} is a consequence of \Cref{T2}. An application of \Cref{C2} is an affirmative answer to a question posed by Zhao concerning quasi $k$-Gorensteinness when the Frobenius extension is between Noetherian algebras; see \Cref{affirmative}. In general, Huang \cite{Huang1999} observed that the notation of quasi $k$-Gorensteinness is not left-right symmetric. However, by applying \Cref{C2}, we could give sufficient conditions that ensure the left-right symmetry of quasi $k$-Gorensteinness; see \Cref{symmetric case}.

\begin{corollary}[See \ref{AD-quasi-Gorenstein}]\label{C2}
     Let $R$ and $A$ be two Noetherian algebras, and $\varphi\colon R\rightarrow A$ be a finite ring homomorphism such that $\varphi$ is a quasi-faithfully flat extension. Suppose that $A$, viewed as a left $R$-module, is Gorenstein projective, and $\Hom_R(A,R)$ is projective as a left $A$-module. Then for each $k > 0$,  the ring $R$ is left quasi $k$-Gorenstein if and only if so is $A$. 
\end{corollary}

In \Cref{Section 6}, we study the finite representation type of $k$-torsionfree modules along a ring homomorphism. Let $\varphi\colon R \rightarrow A$ be a separable split Frobenius extension. Assume that the Krull–Remak–Schmidt theorem holds for both finitely generated left $R$-modules and finitely generated left $A$-modules; this is satisfied, for instance, when $R$ and $A$ are left Artinian rings or Henselian local rings. Under these assumptions, the category of $k$-torsionfree left $R$-modules has finite representation type if and only if the same holds over $A$; see \Cref{asent and descent of finite type of TF}. A similar equivalence holds for CM finiteness; see \Cref{CM-finite result}. 

Let $\Lambda$ be a left Artinian ring, and let $G$ be a finite group acting on $\Lambda$. Suppose that $|G|$ is invertible in $\Lambda$. Then, for each $k > 0$, an application of \Cref{asent and descent of finite type of TF} shows that the category of $k$-torsionfree over the skew group ring $\Lambda G$ has finite representation type if and only if the same holds over $\Lambda$. Furthermore, the skew group ring $\Lambda G$ is CM-finite if and only if $\Lambda$ is CM-finite; see \Cref{skew-application}.

\begin{ack}
The author would like to thank Souvik Dey, Xuesong Lu, and Yaohua Zhang for their discussions related to this work. The author would like to thank the anonymous referee for pointing out the bimodule structure in \Cref{structure} and for providing valuable comments and suggestions, which have improved the paper. The author was supported by the National Natural Science Foundation of China (No. 12401046) and the Fundamental Research Funds for the Central Universities (Nos. CCNU24JC001, CCNU25JC025,  CCNU25JCPT031).
\end{ack}

\section{Notation and Terminology}

In this article, a ring is called Noetherian if it is Noetherian on both sides. Throughout this article, let $R$ be a Noetherian ring. By a module over $R$, we always mean a left $R$-module, and thus a module over $R^{\mathrm{op}}$ is identified with a right $R$-module. The category of finitely generated left $R$-modules is denoted by $\mo(R)$. Let $\proj(R)$ denote the full subcategory of $\mo(R)$ consisting of projective modules.
 For two Noetherian rings $A$ and $R$, an $A\mbox{-}R$-bimodule $M$ will mean a module over both $A$ and $R^{\rm op}$ that such that the actions are compatible: $(a\cdot m)\cdot r=a\cdot (m\cdot r)$ for all $a\in A$, $r\in R$, and $m\in M$. 

For a ring homomorphism $\varphi\colon R\rightarrow A$, the map $\varphi$ is said to be \emph{finite} provided
that $A$ is finitely generated as both a module over $R$ and as a module over $R^{\rm op}$. A ring $A$ is said to be a \emph{Noetherian algebra} if there exists a finite ring homomorphism $\varphi\colon R\rightarrow A$, where $R$ is commutative and the image of $R$ is in the center of $A$. In this case, we will say $A$ is a Noetherian algebra via $\varphi$.

\begin{chunk}\label{def-torsionfree}
  \textbf{Transpose and $k$-torsionfree modules.}     Let $M$ be a module in $\mo(R)$. For each projective resolution of $M$ over $R${\rm :}
       $$
      \epsilon\colon   P_1\xrightarrow{f}P_0\rightarrow  M\to 0,$$ 
where $P_0,P_1\in \proj(R)$. 
It induces a long exact sequence
$$
    0\rightarrow \Hom_R(M,R)\rightarrow \Hom_R(P_0,R)\xrightarrow{\Hom_R(f,R)}\Hom_R(P_1,R)\rightarrow \Tr_R^{\epsilon}(M)\rightarrow 0,
    $$
where the module $\Tr_R^\epsilon(M)$ is called a \emph{transpose} of $M$; see \cite{Auslander-Bridger}. Note that the transpose of $M$ is independent of the choice of the projective resolution of $M$ up to projective summands. If there is no confusion, we will simply write $\Tr_R^\epsilon (M)$ as $\Tr_R(M)$. For each $k>0$, the module $M$ is said to be \emph{$k$-torsionfree} if $\Ext^i_R(\Tr_R(M),R)=0$ for all $1\leq i\leq k$.  

Note that there exists an exact sequence
$$
0\rightarrow  \Ext^1_R(\Tr_R(M),R)\rightarrow M\xrightarrow{\delta_M}M^{\ast\ast}\rightarrow \Ext^2_R(\Tr_R(M),R)\rightarrow  0,
$$
where  $\delta_M(m)(f)=f(m)$ for $m\in M$ and $f\in M^\ast$. It follows that $M$ is $2$-torsionfree if and only if $\delta_M$ is an isomorphism, namely $M$ is reflexive.
\end{chunk}

\begin{chunk}
    \textbf{Syzygy modules.} 
For each $M\in\mo(R)$ and $n \geq 1$, let $\Omega^n_R(M)$ denote the $n$-th syzygy of $M$. That is, there exists an exact sequence in $\mo(R)$:
\[
0 \to \Omega^n_R(M) \to P_{n-1} \to \cdots \to P_1 \to P_0 \to M \to 0,
\]
where $P_i\in\proj(R)$ for $0 \leq i \leq n-1$. It follows from Schanuel's lemma that $\Omega^n_R(M)$ is independent of the choice of the projective resolution of $M$ up to projective summands. 

Let $\pd_R(M)$ denote the projective dimension of $M$ over $R$. Note that $\pd_R(M)\leq n$ if and only if $\Omega^n_R(M)$ is projective.
\end{chunk}

\begin{chunk}\label{gp}
\textbf{Gorenstein projective modules.} An acyclic complex of projective modules over $R$
 $$
 \mathbf{P}  = \cdots \longrightarrow P_{1}\stackrel{\partial_{1}}\longrightarrow P_0\stackrel{\partial_0}\longrightarrow P_{-1}\longrightarrow\cdots
 $$
is called \emph{totally acyclic} provided that $\Hom_R(\mathbf{P}, Q)$ is still acyclic for any projective module $Q$ over $R$. A module $M$ over $R$ is \emph{Gorenstein projective} if there exists a totally acyclic complex $\mathbf{P}$ such that $M$ is isomorphic to the image of $\partial_0$; see details in \cite{EJ}. Any projective  $R$-module is Gorenstein projective. For each totally acyclic complex $\mathbf{P}$, the image of $\partial_i$, denoted by ${\rm Im}(\partial_i)$, is Gorenstein projective for each $i\in \Z$. The full subcategory of $\mo(R)$ consisting of Gorenstein projective modules will be denoted by $\Gproj(R)$. 

For each $M\in\mo(R)$, let $\Gpd_R(M)$ denote the Gorenstein projective dimension of $M$ over $R$.  This invariant, originally introduced as the Gorenstein dimension in \cite[Section 3]{Auslander-Bridger}, measures how far $M$ is from being Gorenstein projective. If $\Gpd_R(M)<\infty$, then it follows from \cite[Remark 3.7]{Auslander-Bridger} that 
$$
\Gpd_R(M)=\sup\{n\geq 0\mid \Ext^n_R(M,R)\neq 0\};
$$
see also \cite[Corollary 2.21]{Holm}.
\end{chunk}

\begin{chunk}\label{def of I-Gorenstein}
\textbf{Iwanaga-Gorenstein rings.}
A Noetherian ring $R$ is said to be \emph{Iwanaga-Gorenstein} if $R$ has finite injective dimension over both $R$ and $R^{\rm op}$. It follows from \cite[Lemma A]{Zaks} that if $R$ is Iwanaga-Gorenstein, then
$$\id_R(R)=\id_{R^{\rm op}}(R)<\infty;$$ here, $\id$ represents the injective dimension. 
A commutative Noetherian ring $R$ is said to be \emph{Gorenstein} if $R_\p$ is Iwanaga-Gorenstein for each prime ideal $\p$ of $R$.
In the case  $\id_R(R)=\id_{R^{\rm op}}(R)=k<\infty$, $R$ is said to be \emph{$k$-Iwanaga-Gorenstein}. 

If $R$ is Iwanaga-Gorenstein, then each finitely generated module over $R$ has finite Gorenstein projective dimension.
\end{chunk}

\begin{chunk}\label{def of Frobenius}
 \textbf{Frobenius extensions.}   The notion of the Frobenius extension is a generalization of Frobenius algebra \cite{Kas}. Following \cite[Theorem 1.2]{Kadison},  a ring homomorphism $R\rightarrow A$ is called a \emph{Frobenius extension} if the following equivalent conditions hold:
\begin{enumerate}
    \item $A$ is a finitely generated projective  over $R$ and there is an isomorphism $A\cong \Hom_{R}(A,R)$ as $A\mbox{-}R$-bimodules.

    \item $A$ is a finitely generated projective over $R^{\rm op}$ and there is an isomorphism $A\cong \Hom_{R^{\rm op}}(A,R)$ as $R\mbox{-}A$-bimodules.
\end{enumerate}
Let $G$ be a finite group, the injection $\Z\rightarrow \Z G$ is a classical example of a Frobenius extension.

 Let $k$ be a field and $A$ be a $k$-algebra via a ring homomorphism $\varphi\colon k\rightarrow A$. Then $A$ is a Frobenius $k$-algebra if and only if $\varphi$ is a Frobenius extension; see \cite[Definition 4.2.5]{Weibel}.
\end{chunk}

\begin{chunk}
\textbf{Derived categories.} Let $\D(R)$ denote the derived category of complexes of modules over $R$. The category $\D(R)$ is a triangulated category with the suspension functor $[1]$; for each complex $X$, $X[1]_i\colonequals X_{i-1}$, and  $\partial_{X[1]}\colonequals -\partial_X$.  The bounded derived category $\D^f_b(R)$ is the full subcategory of $\D(R)$ consisting
of complexes with finitely generated total homology

Recall that a complex $X$ over $R$ is called \emph{homotopy projective} (resp. \emph{homotopy injective}) provided that $\Hom_R(X,-)$ (resp. $\Hom_R(-,X)$) preserves acyclic complexes. See \cite[Section 3]{Spaltenstein} for the existence of the homotopy projective resolution and the homotopy injective resolution of complexes. 

Let $\RHom_R(-,-)$ denote the right derived functor of $\Hom_R(-,-)$. For each $M, N$ in $\D(R)$, the complex $\RHom_R(M, N)$ can be represented by either $\Hom_R(P,N)$ or $\Hom_R(M,I)$, where $P\xrightarrow \simeq M$ is a homotopy projective resolution and $N\xrightarrow\simeq I$ is a homotopy injective resolution. 
Note that for modules $M,N$ over $R$, $$\Ext^i_R(M,N)=\h_{-i}\RHom_R(M,N).$$
\end{chunk}

\section{Base change of (Gorenstein) transpose}\label{Section 3}

In this section, we investigate the behavior of the (Gorenstein) transpose along a ring homomorphism. The main result is Theorem \ref{T1} from the introduction; see Theorem \ref{essential}. Throughout, $R$ will be a Noetherian ring.

\begin{chunk}\label{Gorenstein transpose}
     For each $M\in\mo(R)$ and a Gorenstein projective resolution of $M$ over $R$:
    $$
   \pi\colon  G_1\xrightarrow g G_0\rightarrow M\rightarrow 0,
    $$
   where $G_0,G_1\in\Gproj(R)$. It induces a long exact sequence over $R^{\rm op}$:
    $$
    0\rightarrow \Hom_R(M,R)\rightarrow \Hom_R(G_0,R)\xrightarrow{\Hom_R(g,R)}\Hom_R(G_1,R)\rightarrow \Tr_R^{\G,\pi}(M)\rightarrow 0,
    $$
     where the module $\Tr_R^{\G,\pi}(M)$, introduced by Huang and Huang \cite[Section 3]{Huang-Huang}, is called a \emph{Gorenstein transpose} of $M$; compare transpose in \Cref{def-torsionfree}. The definition of the Gorenstein transpose depends on the choice of the Gorenstein projective resolution of $M$. When the context is clear,  $\Tr_R^{\G,\pi}(M)$ will be simply written as $\Tr_R^{\G}(M)$.

    A transpose is a Gorenstein transpose. The converse is not true in general; see \cite{Huang-Huang}. However, for each $M\in\mo(R)$, it is proved in \cite[Theorem 3.1]{Huang-Huang} that any Gorenstein transpose of $M$ can be embedded into a transpose of $M$ with the cokernel Gorenstein projective. In particular, for each $i>0$,
    $$
    \Ext^i_{R^{\rm op}}(\Tr_R(M),R)\cong \Ext^i_{R^{\rm op}}(\Tr_R^{\G}(M),R).
    $$
\end{chunk}
\begin{chunk}\label{structure}
Let $A,B$ and $C$ be rings. For an $A\mbox{-}B$-bimodule $M$ and an $A\mbox{-}C$-bimodule $N$, $\Ext^n_A(M,N)$ is a $B\mbox{-}C$-bimodule for each $n\geq 0$. In particular, for a ring homomorphism $R\to A$, $\Ext_R^n(A,R)$ is an $A\mbox{-}R$-bimodule.

Indeed, choose an injective resolution $N\xrightarrow{\simeq} I$  of $N$ over $A$. For each $c\in C$, the right multiplication $r_c\colon N\rightarrow N; x\mapsto xc$ is $A$-linear. Lift $r_c$ to a chain map $\tilde{r_c}\colon I\rightarrow I$. The right $C$-action by $c$ on $\Ext^n_A(M,N)=\h_{-n}\Hom_A(M,I)$ is given by $\h_{-n}\Hom_A(M,\tilde{r_c})$. Similarly, there is a left $B$-action on $\Ext^n_A(M,N)$ via the projective resolution of $M$, and $\Ext^n_A(M,N)$ carries a $B\mbox{-}C$-bimodule structure.


\end{chunk}
Next, we present several lemmas used in the proof of \Cref{essential}.
\begin{lemma}\label{commutative-finite algebra}
    Let $R\rightarrow A$ be a finite ring homomorphism. Let $P_1\rightarrow P_0\rightarrow M\rightarrow 0$ be a projective resolution of $M$ in $\mo(A)$, where $P_0,P_1\in\proj(A)$. For each $n\geq 0$, there exists an exact sequence in $\mo(R^{\rm op})${\rm :}
    $$
   \Ext^n_R(P_0, R)\rightarrow \Ext^n_R(P_1,R)\rightarrow \Tr_A(M)\otimes_A \Ext^n_R(A,R)\rightarrow 0.
    $$
\end{lemma}

\begin{proof}
 Let $I$ be an injective resolution of $R$. Then $\Ext^n_R(A,R)=\h_{-n}\Hom_R(A,I)$. For each finitely generated projective module $P$ over $A$, since $\Hom_A(P,-)$ is exact, we get the first isomorphism below:
\begin{align*}
    \Hom_A(P,\Ext^n_R(A,R))& =\Hom_A(P,\h_{-n}\Hom_R(A,I))\\
    & \cong \h_{-n}\Hom_A(P,\Hom_R(A,I))\\
    & \cong \h_{-n}\Hom_R(P,I)\\
    & =\Ext^n_R(P,R),
\end{align*}
where the second isomorphism is due to the adjunction $({\rm Res}, \Hom_R(A,-))$.
Thus, we get the isomorphisms in the following commutative diagram in $\mo(R^{\rm op})${\rm :}
    $$
    \xymatrix{
\Hom_A(P_0,\Ext^n_R(A,R)) \ar[r]\ar[d]^-\cong &  \Hom_A(P_1,\Ext^n_R(A,R))\ar[r]\ar[d]^-\cong & C\ar[r]\ar@{-->}[d]^-\exists& 0\\
\Ext^n_R(P_0,R)\ar[r]& \Ext^n_R(P_1,R)\ar[r]&     C^\prime\ar[r]& 0.
    }
    $$
 This yields that $C\cong C^\prime$. By \cite[Proposition 4.3]{Takahashi2013}, $$\Tr_A(M)\otimes_A\Ext^n_R(A,R)\cong C.$$ Thus, we conclude that $\Tr_A(M)\otimes_A\Ext^n_R(A,R)\cong C^\prime$.


\end{proof}

\begin{lemma}\label{kernel}
    Let $R\rightarrow A$ be a finite ring homomorphism. Assume there exists $n\geq 0$ such that $\Ext^i_R(A,R)=0$ for all $i< n$. Then{\rm:}
    \begin{enumerate}
        \item For each $M\in\mo(A)$, $\Ext^i_A(M,R)=0$ for all $i< n$. 

        \item Let $P_1\rightarrow P_0\rightarrow M\rightarrow 0$ be a projective resolution of $M$ in $\mo(A)$, where $P_0,P_1\in\proj(A)$. Then it induces an exact sequence in $\mo(R^{\rm op})${\rm :}
        $$
        0\rightarrow \Ext^n_R(M,R)\rightarrow \Ext^n_R(P_0,R)\rightarrow \Ext^n_R(P_1,R)\rightarrow \Tr_A(M)\otimes_A\Ext^n_R(A,R)\rightarrow 0.
        $$
    \end{enumerate}
\end{lemma}
\begin{proof}

 (1)  Set $(-)^\ast=\Hom_R(-,R)$. Choose a short exact sequence
    $$
    0\rightarrow N\rightarrow P\rightarrow M\rightarrow 0
    $$
    in $\mo(A)$, where $P\in\proj(A)$.  Since $\Ext^i_R(A,R)=0$ for all $i<n$, we get that $\Ext^i_R(P,R)=0$ for all $i<n$. Combining with this, by applying $\Hom_R(-,R)$ to the above short exact sequence, we get a long exact sequence
    \begin{center}
     $
    0\rightarrow M^\ast\rightarrow 0\rightarrow N^\ast\rightarrow \Ext^1_R(M,R)\rightarrow 0\rightarrow \Ext^1_R(N,R)\rightarrow \Ext^2_R(M,R)\rightarrow 0\rightarrow \cdots\rightarrow \Ext^{n-2}_R(M,R)\rightarrow 0\rightarrow \Ext^{n-2}_R(N,R)\rightarrow \Ext^{n-1}_R(M,R)\rightarrow 0.
    $
    \end{center}
It follows that $M^\ast=0$, and hence $N^\ast=0$ as $M\in \mo(A)$ is arbitrary. Combining this with the long exact sequence, we have $\Ext^1_R(M,R)=0$, and hence $\Ext^1_R(N,R)=0$ as well. By the same argument, we deduce that $\Ext^i_R(M,R)=0$ for $i<n$. 

(2) Consider the short exact sequence in $\mo(A)$:
$$
0\rightarrow \Omega_A^1(M)\rightarrow P_0\rightarrow M\rightarrow 0.
$$
Combining this with (1), by applying $\Hom_R(-,R)$ to the above short exact sequence, there exists an exact sequence in $\mo(R^{\rm op})$:
\begin{equation}\label{an exact sequence}\tag{$\dagger$}
    0\rightarrow \Ext^n_R(M,R)\rightarrow \Ext^n_R(P_0,R)\xrightarrow\alpha \Ext^n_R(\Omega^1_A(M),R).
\end{equation}
Note that there is also a short exact sequence in $\mo(A)$:
$$
0\rightarrow \Omega_A^2(M)\rightarrow P_1\rightarrow \Omega^1_A(M)\rightarrow 0.
$$
The same reason as (\ref{an exact sequence}), there is an injection $\Ext^n_R(\Omega^1_A(M),R)\xrightarrow {\iota}\Ext^n_R(P_1,R)$. Since $\iota$ is injective, combining with (\ref{an exact sequence}), we get an exact sequence
$$
0\rightarrow \Ext^n_R(M,R)\rightarrow \Ext^n_R(P_0,R)\xrightarrow{\iota\circ \alpha} \Ext^n_R(P_1,R).
$$
Note that the map $\iota\circ \alpha$ is induced by the map $P_1\rightarrow P_0$. 
By \Cref{commutative-finite algebra}, the cokernel of $\iota\circ \alpha$ is isomorphic to $\Tr_A(M)\otimes_A\Ext^n_R(A,R)$. 
\end{proof}

\begin{lemma}\label{long exact sequence}
   For each $M\in\mo(R)$, choose an exact sequence
    $$
 0\rightarrow \Omega^n_R(M)\rightarrow P_{n-1}\rightarrow P_{n-2}\rightarrow \cdots \rightarrow P_1\rightarrow P_0\rightarrow M\rightarrow 0,
    $$
    where $P_i\in \proj(R)$ for $0\leq i\leq n-1$. 
    Assume there exists $n> 0$ such that $\Ext^i_R(M,R)=0$ for all $i<n$. Then the above exact sequence induces a long exact sequence in $\mo(R^{\rm op})${\rm :}
        $$
     0\rightarrow (P_0)^\ast\rightarrow (P_1)^\ast\rightarrow \cdots  \rightarrow (P_{n-2})^\ast\rightarrow (P_{n-1})^\ast\rightarrow  (\Omega^n_R(M))^\ast\rightarrow \Ext^n_R(M,R)\rightarrow 0,
        $$
where $(-)^\ast\colonequals\Hom_R(-,R)$.
 In particular, if $n\geq 2$, then there exists a short exact sequence 
    $$
   0\rightarrow \Tr_R(\Omega^{n-2}_R(M))\rightarrow (\Omega_R^n(M))^\ast\rightarrow \Ext^n_R(M,R)\rightarrow 0. 
    $$
\end{lemma}
\begin{proof}
  
    For each $i<n$, consider the short sequence in $\mo(R)$:
    $$
 0\rightarrow    \Omega^{i+1}_R(M)\rightarrow P_i\rightarrow \Omega^i_R(M)\rightarrow 0.
    $$
    This induces an exact sequence in $\mo(R^{\rm op})$:
    $$
    0\rightarrow (\Omega_R^i(M))^\ast\rightarrow (P_i)^\ast\rightarrow (\Omega^{i+1}_R(M))^\ast\rightarrow \Ext^1_R(\Omega^i_R(M),R)\rightarrow 0
    $$
    Note that for each $i\geq 0$, there is an isomorphism
    $$\Ext^1_R(\Omega^i_R(M),R)\cong \Ext^{i+1}_R(M,R).$$ This is equal to $0$ if $i<n-1$. Thus, for each $i<n-1$, we get a short exact exact sequence in $\mo(R^{\rm op})$:
    $$
    0\rightarrow (\Omega_R^i(M))^\ast\rightarrow (P_i)^\ast\rightarrow (\Omega^{i+1}_R(M))^\ast\rightarrow 0,
    $$
    and for $i=n-1$, we get an exact sequence in $\mo(R^{\rm op})$:
    $$
    0\rightarrow (\Omega_R^{n-1}(M))^\ast\rightarrow (P_{n-1})^\ast\rightarrow (\Omega^{n}_R(M))^\ast\rightarrow \Ext^n_R(M,R)\rightarrow 0.
    $$
  The desired long exact sequence in the lemma can be obtained by combining these exact sequences. This completes the proof.
\end{proof}

\begin{chunk}\label{snake}
Assume there is a commutative diagram  in $\mo(R)$ with exact rows and exact columns:
$$
\xymatrix{
& 0 \ar[d]& 0\ar[d]& 0\ar[d] &0\ar[d]\\
0\ar[r]& M_1\ar[r]\ar[d]^-{f_1}& M_2\ar[r]\ar[d]^-{f_2}& M_3\ar[d]^-{f_3}\ar[r]& M_4\ar[r]\ar[d]^-{f_4}& 0\\
0\ar[r]& N_1\ar[r]& N_2\ar[r]& N_3\ar[r]& N_4\ar[r]& 0.
}
$$
By using the Snake lemma, the above diagram induces a long exact sequence
$$
0\rightarrow \Coker(f_1)\rightarrow \Coker(f_2)\rightarrow \Coker(f_3)\rightarrow \Coker(f_4)\rightarrow 0.
$$
\end{chunk}

\begin{theorem}\label{essential}
       Let $\varphi\colon R\rightarrow A$ be a finite ring homomorphism, where $R$ is a Noetherian ring. Assume $\Gpd_R(A)<\infty$ and there exists $n\geq 0$ such that $\Ext^i_R(A,R)=0$ for all $i\neq n$. For each $M\in\mo(A)$ and the transpose $\Tr_A(M)$, we have{\rm:}
    \begin{enumerate}
        \item  If $n=0$, then there exists an isomorphism  in $\mo(R^{\rm op})${\rm :}
        $$
 \Tr^G_R(M)\cong \Tr_A(M)\otimes_A \Hom_R(A,R). 
 $$
 
\item  If $n=1$,  then there exists a short exact sequence in $\mo(R^{\rm op})${\rm :}
$$
0\rightarrow Q\rightarrow \Tr_R^G(\Omega^1_R(M))\rightarrow \Tr_A(M)\otimes_A \Ext^1_R(A,R)\rightarrow 0,
$$
 where $Q$ is a finitely generated projective over $R$. 
        \item If $n\geq 2$, then there exists a short exact sequence in $\mo(R^{\rm op})${\rm :}
       $$
       0\rightarrow \Tr_R(\Omega^{n-2}_R(\Omega^2_A(M)))\rightarrow \Tr^G_R(\Omega^{n}_R(M))\rightarrow \Tr_A(M)\otimes_A \Ext^n_R(A,R)\rightarrow 0
       $$
       with $\pd_R(\Tr_R(\Omega^{n-2}_R(\Omega^2_A(M))))
       \leq n-1$.
    \end{enumerate}
    
\end{theorem}
\begin{proof}
By \Cref{gp}, we have $\Gpd_R(A)=n$. The statement of (1) now follows from \Cref{kernel}(2). The proof of the statement (2) is similar to that of (3) and easier. Therefore, we will prove statement (3) as an example.

(3)    For the case of $n\geq 2$. Choose an exact sequence in $\mo(A)$:
    $$
  0\rightarrow \Omega^2_A(M)\rightarrow P^1\rightarrow P^0\rightarrow M\rightarrow 0,
    $$
where each $P^0,P^1\in\proj(A)$. 

Consider the above exact sequence as an exact sequence over $R$. By the Horseshoe lemma, we get the following commutative diagram in $\mo(R)$ with exact rows and exact columns${\rm:}$
    $$
    \xymatrix{
       & 0\ar[d]& 0\ar[d]& 0\ar[d]& 0\ar[d]&\\
0\ar[r]& \Omega^n_R(\Omega^2_A(M))\ar[r]\ar[d]& \Omega^n_R(P^1)\ar[r]\ar[d]& \Omega^n_R(P^0)\ar[r]\ar[d]& \Omega^n_R(M)\ar[r]\ar[d]& 0\\
    0\ar[r]& P^2_{n-1}\ar[r] \ar[d]& P^1_{n-1}\ar[r]\ar[d]& P^0_{n-1}\ar[r]\ar[d]& P^M_{n-1}\ar[r]\ar[d]& 0\\
      0\ar[r]& P^2_{n-2}\ar[r] \ar[d]& P^1_{n-2}\ar[r]\ar[d]& P^0_{n-2}\ar[r]\ar[d]& P^M_{n-2}\ar[r]\ar[d]& 0\\
    & \vdots\ar[d]& \vdots\ar[d]& \vdots\ar[d]& \vdots\ar[d]& \\
    0\ar[r]& P^2_1\ar[r] \ar[d]& P^1_1\ar[r]\ar[d]& P^0_1\ar[r]\ar[d]& P^M_1\ar[r]\ar[d]& 0\\
0\ar[r]& P^2_0\ar[r] \ar[d]& P^1_0\ar[r]\ar[d]& P^0_0\ar[r]\ar[d]& P^M_0\ar[r]\ar[d]& 0\\
    0\ar[r]& \Omega^2_A(M)\ar[r]\ar[d]& P^1\ar[r]\ar[d]& P^0\ar[r] \ar[d]& M\ar[r]\ar[d]& 0\\
    & 0& 0& 0& 0&
},
    $$
    where  $P^i_j$ and  $P^M_j$ are finitely generated projective over $R$ for $0\leq i\leq 2, 0\leq j\leq n-1$. Combining this with \Cref{kernel} and \Cref{long exact sequence}, by applying the functor $(-)^\ast\colonequals\Hom_R(-,R)$ to the above commutative diagram, we can obtain the following commutative diagram in $\mo(R^{\rm op})$ with exact rows and exact columns${\rm:}$
    $$
    \xymatrix{
       & 0\ar[d]& 0\ar[d]& 0\ar[d]& 0\ar[d]&\\
     0\ar[r]& (P^M_{0})^\ast\ar[r] \ar[d]& (P^0_{0})^\ast\ar[r]\ar[d]& (P^1_{0})^\ast\ar[r]\ar[d]& (P^2_{0})^\ast\ar[r]\ar[d]& 0\\
       0\ar[r]& (P^M_{1})^\ast\ar[r] \ar[d]& (P^0_{1})^\ast\ar[r]\ar[d]& (P^1_{1})^\ast\ar[r]\ar[d]& (P^2_{1})^\ast\ar[r]& 0\\
    & \vdots\ar[d]& \vdots\ar[d]& \vdots\ar[d]& \vdots\ar[d]& \\
       0\ar[r]& (P^M_{n-2})^\ast\ar[r] \ar[d]^-{\alpha_1}& (P^0_{n-2})^\ast\ar[r]\ar[d]^-{\alpha_2}& (P^1_{n-2})^\ast\ar[r]\ar[d]^-{\alpha_3}& (P^2_{n-2})^\ast\ar[r]\ar[d]^-{\alpha_4}& 0\\
    0\ar[r]& (P^M_{n-1})^\ast\ar[r] \ar[d]& (P^0_{n-1})^\ast\ar[r]^-f\ar[d]& (P^1_{n-1})^\ast\ar[r]\ar[d]& (P^2_{n-1})^\ast\ar[r]& 0\\
0\ar[r]& (\Omega^n_R(M))^\ast\ar[r] \ar[d]& (\Omega^n_R(P^0))^\ast\ar[r]^-g\ar[d]& (\Omega^n_R(P^1))^\ast\ar[d]& & \\
    0\ar[r]& \Ext^n_R(M,R)\ar[r]\ar[d]& \Ext^n_R(P^0,R)\ar[r]^-h\ar[d]& \Ext^n_R(P^1,R) \ar[d]& &\\
    & 0& 0& 0& &
    }.
    $$
By repeatedly applying \Cref{snake} to the diagram above, we obtain an induced exact sequence involving the cokernels $\Coker(\alpha_i)$. Namely, there is an exact sequence
$$
0\rightarrow \Tr_R(\Omega_R^{n-2}(M))\rightarrow \Tr_R(\Omega_R^{n-2}(P^0))\xrightarrow {\overline{f}}\Tr_R(\Omega_R^{n-2}(P^1))\rightarrow \Tr_R(\Omega_R^{n-2}(\Omega^2_A(M)))\rightarrow 0.
$$
Note that the projective dimension of each $\Coker(\alpha_i)$ is at most $n-1$. In particular, $\pd_R (\Tr_R(\Omega_R^{n-2}(\Omega^2_A(M))))\leq n-1$.

The two middle exact columns in the diagram above induce a commutative diagram in $\mo(R^{\rm op})${\rm :}
$$
\xymatrix{
0\ar[r] & \Tr_R(\Omega^{n-2}_R(P^0))\ar[r]\ar[d]^-{\overline{f}} & (\Omega^n_R(P^0))^\ast\ar[r]\ar[d]^-g& \Ext^n_R(P^0,R)\ar[r]\ar[d]^-h& 0\\
0\ar[r] & \Tr_R(\Omega^{n-2}_R(P^1))\ar[r]& (\Omega^n_R(P^1))^\ast\ar[r]& \Ext^n_R(P^1,R)\ar[r]& 0\\
}.
$$
By the above two commutative diagrams, the induced map from the kernel of $g$ to the kernel of $h$ is surjective. Combining this with the Snake lemma, we obtain a short exact sequence
$$
0\rightarrow \Coker({\overline{f}})\rightarrow \Coker(g)\rightarrow \Coker(h)\rightarrow 0.
$$
As the above long exact sequence shows, $\Coker({\overline{f}}) \cong \Tr_R(\Omega^{n-2}_R(\Omega_A^2(M)))$. It follows from \Cref{commutative-finite algebra} that $\Coker(h)\cong \Tr_A(M)\otimes_A\Ext^n_R(A,R)$ in $\mo(R^{\rm op})$. As mentioned at the beginning, $\Gpd_R(A) = n$. It follows that $\Omega^n_R(P^0)$ and $\Omega^n_R(P^1)$ are Gorenstein projective over $R$, and hence $\Coker(g)$ is isomorphic to a Gorenstein transpose $\Tr^G_R(\Omega^n_R(M))$. Thus, we obtain the desired short exact sequence.
\end{proof}
\begin{remark}\label{compare}
    (1) Let $\Tr_R^G(M)$ be a Gorenstein transpose of a module $M\in\mo(R)$. As mentioned in \ref{Gorenstein transpose}, Huang and Huang \cite[Theorem 3.1]{Huang-Huang} observed that there exists a short exact sequence
     $$
    0\rightarrow \Tr_R^G(M)\rightarrow\Tr_R(M) \rightarrow H\rightarrow 0
    $$
with $H$ Gorenstein projective. Motivated by this, Zhao and Sun \cite[Theorem 2.3]{Zhao-Sun} proved that  there exists a short exact sequence
    $
    0\rightarrow H^\prime\rightarrow \Tr_R^G(M)\rightarrow \Tr_R(M)\rightarrow 0
    $
    with $H^\prime$ Gorenstein projective.

    (2)  Our result is quite different from the above results; note that Theorem \ref{essential} is trivial if $R=A$.  
   Let $\varphi\colon R\rightarrow A$ be a ring homomorphism, where $R$ is a  Noetherian ring. Assume $\Gpd_R(A)<\infty $,  and there exists $n\geq 0$ such that $\Ext^i_R(A,R)=0$ for $i\neq n$ and $\Ext^n_R(A,R)\cong A$ in $\mo(A)$. Then $\Tr_A(M)\otimes_A\Ext^n_R(A,R)\cong \Tr_A(M)$. Hence, for each transpose $\Tr_A(M)$ of $M\in\mo(A)$, Theorem \ref{essential} yields:
    \begin{itemize}
        \item Assume $n=0$. $\Tr_A(M)$ is isomorphic to some Gorenstein transpose $\Tr_R^G(M)$.

        \item Assume $n\geq 1$, there exists a Gorenstein transpose $\Tr_R^G(\Omega^n_R(M))$ and a short exact sequence
        $
        0\rightarrow K\rightarrow \Tr_R^G(\Omega^n_R(M))\rightarrow \Tr_A(M)\rightarrow 0
        $
        with $\pd_R(K)\leq n-1$.
    \end{itemize}
    \end{remark}
\begin{chunk}\label{def-of-ad-gdim}
Similar to the notion of Gorenstein projective dimension for a module, there is a corresponding concept of Gorenstein dimension for complexes in the bounded derived category; see \cite[Definition 2.3.2]{Christensen-book}. Let $\varphi\colon R \rightarrow A$ be a finite ring homomorphism. Following \cite[Definition 3.1]{Liu-Ren}, the map $\varphi$ is said to have the \emph{ascent and descent of finite Gorenstein dimension} if the following two conditions hold:
\begin{enumerate}
    \item A complex in $\D^f_b(A)$ has finite Gorenstein dimension over $A$ if and only if it has finite Gorenstein dimension over $R$;
    \item A complex in $\D^f_b(A^{\mathrm{op}})$ has finite Gorenstein dimension over $A^{\mathrm{op}}$ if and only if it has finite Gorenstein dimension over $R^{\mathrm{op}}$.
\end{enumerate}
This holds, for example, when $\varphi$ is a finite local homomorphism between commutative Gorenstein local rings. For other classes of ring homomorphisms that satisfy the above property, see \cite{Liu-Ren}.

If $R$ is commutative and $A$ is an $R$-algebra via $\varphi$, then an equivalent characterization of the above property is given in \cite[Theorem 3.12]{Liu-Ren}. Moreover, the notion of \emph{ascent and descent of the Gorenstein projective property} is also investigated in \cite{Liu-Ren}.
\end{chunk}

    \begin{chunk}
    Let $\varphi\colon R\rightarrow A$ be a finite ring homomorphism.  In the subsequent sections, we will consider the following condition:
    $$
    \RHom_R(A,R)\cong P[-n] \text{ in }\D(A) \text{ for some } P\in\proj(A) \text{ and }n\geq 0.
    $$
    This condition is equivalent to that there exists $n\geq 0$ such that $\Ext^i_R(A,R)=0$ for $i\neq n$ and $\Ext^n_R(A,R)$ is finitely generated projective over $A$. Note that this condition is stronger than the assumption of \Cref{essential}.
\end{chunk}

\begin{example}\label{AD-Gorenstein dimension}
   (1)
   Let $\varphi\colon R\rightarrow A$ be a Frobenius extension; see \ref{def of Frobenius}.
   Then $A$ is projective over $R$ on both sides and $\Hom_R(A,R)\cong A$ as $A\mbox{-}R$-bimodules.

  (2)   Let \(\varphi\colon R \to A\) be a finite ring homomorphism between commutative Noetherian rings. Assume that \(A\) is local and that \(\varphi\) has the ascent and descent of finite Gorenstein dimension property (\Cref{def-of-ad-gdim}). Then  $\Gpd_R(A)=n$ and
$
\RHom_R(A, R) \cong A[-n] \quad \text{in } \D(A) \text{ for some } n \geq 0.
$

Indeed, since $\varphi$ has ascent and descent of finite Gorenstein dimension property, it follows from \cite[Theorem 3.11]{Liu-Ren} that $\Gpd_R(A)<\infty$ and $\RHom_R(A,R)$ is perfect over $A$.  Then, by \cite[Theorem 6.1 and Proposition 8.3]{Christensen2001}, one has an isomorphism $\RHom_R(A,R)\cong A[-n]$ in $\D(A)$ for some $n\geq 0$. The equality $\Gpd_R(A)=n$ follows from \Cref{gp}.
\end{example}

\begin{corollary}
    Let $\varphi\colon R\rightarrow A$ be a finite ring homomorphism, where $R$ is a Noetherian ring. Assume $\Gpd_R(A)<\infty $ and $\RHom_R(A,R)\cong A[-n]$ in $\D(A)$ for some $n\geq 0$. For each $M\in\mo(A)$, there exists a commutative diagram in $\mo(R^{\rm op})$ with exact rows and exact columns
$$
\xymatrix{
& & 0\ar[d] & 0\ar[d]& &\\
0\ar[r]& K\ar[r]\ar@{=}[d] &\Tr^G_R(\Omega^n_R(M))\ar[r]\ar[d]& \Tr_A(M)\ar[r]\ar[d]& 0\\
0\ar[r]& K\ar[r] &\Tr_R(\Omega^n_R(M))\ar[r]\ar[d] & W\ar[r]\ar[d]& 0\\
& & C \ar@{=}[r]\ar[d]& C\ar[d]& \\
& & 0 & 0& &
},
$$
where $C$ is Gorenstein projective over $R$ and $\pd_R(K)\leq n-1$. In particular, there exists a short exact sequence in $\mo(R^{\rm op})${\rm :}
$$
0\rightarrow \Tr^G_R(\Omega^n_R(M))\rightarrow \Tr_R(\Omega^n_R(M))\oplus\Tr_A(M)\rightarrow W\rightarrow 0.
$$

\end{corollary}
\begin{proof}
    By Remark \ref{compare}, we get the following short exact sequence in the row with $\pd_R(K)\leq n-1$, and the short exact sequence in the column with $C$ Gorenstein projective:
    $$
\xymatrix{
& & 0\ar[d] & & &\\
0\ar[r]& K\ar[r] &\Tr^G_R(\Omega^n_R(M))\ar[r]^-{f_1}\ar[d]^-{f_2}& \Tr_A(M)\ar[r]& 0\\
&  &\Tr_R(\Omega^n_R(M)\ar[d] & & \\
& & C \ar[d]&& \\
& & 0 & & &
}.
$$ The desired diagram is obtained by taking the pushout of the morphisms $f_1$ and $f_2$.
\end{proof}

\section{Base change of $k$-torsionfree modules}\label{Section 4}
The main result of this section is \Cref{AD}, which implies \Cref{C1} from the introduction. Its proof relies on \Cref{essential}.

\begin{lemma}\label{duality}
     Let $\varphi\colon R\rightarrow A$ be a finite ring homomorphism, where $R$ is a commutative Noetherian ring. Assume $\Gpd_R(A)<\infty$, and there exists an isomorphism $
    \RHom_R(A,R)\cong P[-n] \text{ in }\D(A) \text{ for some } P\in\proj(A) \text{ and }n\geq 0.
    $ Then there exists an isomorphism in $\mo(A^{\rm op}){\rm:}$
        $$
      A\cong \Ext_{R}^n(\Ext^n_R(A,R),R).
        $$

\end{lemma}
\begin{proof}
Choose an injective resolution $R\xrightarrow \simeq I$.  In what follows, we identify the derived functor $\RHom_R(-,R)$ with $\Hom_R(-,I)$. By computing $\Ext^n_R(-,R)$ via $\h_{-n}\Hom_R(-,I)$, the module $\Ext^n_R(A,R)$ carries the structure of an $A\mbox{-}R$-bimodule, and hence $\Ext^n_R(\Ext^n_R(A,R),R)$ has the structure of an $R\mbox{-}A$-bimodule. By assumption, $\Ext^i_R(A,R)= 0$ for all $i\neq n$. It follows that $\Hom_R(A,I)\cong \Ext^n_R(A,R)[-n]$ in $\D(A)$. This yields that the second isomorphism in $\D(A^{\rm op})$ below:
    \begin{align*}
     A&\cong \RHom_{R}(\RHom_R(A,R),R)\\
    & \cong\RHom_{R}(\Ext^n_R(A,R)[-n], R)\\
 & \cong \RHom_{R}(\Ext^n_R(A,R), R)[n]\\
 & \cong \Ext^n_{R}(\Ext^n_R(A,R), R)[-n][n]\\
 & \cong \Ext^n_{R}(\Ext^n_R(A,R), R),
    \end{align*}
 where the first one is because    $\Gpd_R(A)<\infty$ (see \cite[Corollary 2.3.8]{Christensen-book}), 
and the fourth one is due to the assumption that $\Ext^i_R(A,R)=0$ for $i\neq n$ and $\Ext^n_R(A,R)$ is projective over $A$. Thus, we get that $A$ and $\Ext^n_R(\Ext^n_R(A,R),R)$ are isomorphic in $\D(A^{\rm op})$. Since the canonical functor $\mo(A^{\rm op})\rightarrow \D(A^{\rm op})$ is fully faithful,  $A$ and $\Ext^n_R(\Ext^n_R(A,R),R)$ are isomorphic in $\mo(A^{\rm op})$. 
\end{proof}

   \begin{chunk}\label{Ext-isomorphism}
       Let $R_1,R_2$ be Noetherian rings, and $M$ be an $R_1\mbox{-}R_2$-bimodule such that $M$ is flat over $R_1$. Then the functor $-\otimes_{R_1}M$ is a triangle functor from $\D(R_1^{\rm op})$ to $\D(R_2^{\rm op})$, with right adjoint $\RHom_{R_2^{\rm op}}(M,-)$. 
       
       For each $X\in\mo(R_1^{\rm op})$, $Y\in\mo(R_2^{\rm op})$, if $\Ext^n_{R_2^{\rm op}}(M,Y)=0$ for all $n>0$, namely $\RHom_{R^{op}_2}(M,Y)\cong \Hom_{R_2^{\rm op}}(M,Y)$ in $\D(R_1^{\rm op})$, then for each $i\geq 0$, there is an isomorphism:
       $$
       \Ext^i_{R_2^{\rm op}}(X\otimes_{R_1} M,Y)\cong \Ext^i_{R_1^{\rm op}}(X,\Hom_{R_2^{\rm op}}(M,Y)).
       $$
   \end{chunk}
   \begin{lemma}\label{changeofring-iso}
       Let $\varphi\colon R\rightarrow A$ be a finite ring homomorphism. Assume $\Gpd_R(A)<\infty$, and  $
    \RHom_R(A,R)\cong P[-n] \text{ in }\D(A) \text{ for some } P\in\proj(A) \text{ and }n\geq 0.
    $ For each $M\in\mo(A)$ and $i>0$, we have{\rm :}
    \begin{enumerate}
        \item  Assume $n=0$. Then $\Ext^i_{R^{\rm op}}(\Tr_R(M),R)\cong \Ext^i_{A^{\rm op}}(\Tr_A(M),A)$.

\item  Assume, in addition, $n\geq 1$ and $R$ is commutative. Then
$$\Ext^{n+i}_{R}(\Tr_R(\Omega^n_R(M)),R)\cong \Ext^i_{A^{\rm op}}(\Tr_A(M),A).$$
\end{enumerate}
\end{lemma}
\begin{proof}
By \ref{gp} and the assumption, $\Gpd_R(A)=n$.

(1) By \Cref{essential}(1), we get the second isomorphism below:
\begin{align*}
    \Ext^i_{R^{\rm op}}(\Tr_R(M),R)
    & \cong \Ext^i_{R^{\rm op}}(\Tr^G_R(M),R)\\
    & \cong \Ext^i_{R^{\rm op}}(\Tr_A(M)\otimes_A\Hom_R(A,R),R)\\
    & \cong \Ext^i_{A^{\rm op}}(\Tr_A(M),\Hom_{R^{\rm op}}(\Hom_R(A,R),R))\\
    & \cong \Ext^i_{A^{\rm op}}(\Tr_A(M),A)
\end{align*}
where the first isomorphism follows from \ref{Gorenstein transpose}, the third one follows from \Cref{Ext-isomorphism} as $\Hom_R(A, R)$ is projective over $A$ and that $\Ext^j_{R^{\rm op}}(\Hom_R(A, R), R) = 0$ for all $j > 0$ (the vanishing here holds because $\Hom_R(A, R)$ is Gorenstein projective over $R^{\rm op}$, given that $A$ is Gorenstein projective over $R$), and the last one is due to the $A^{\rm op}$-linear isomorphism
$
A\cong \Hom_{R^{\rm op}}(\Hom_R(A,R),R);
$
the isomorphism here holds because $A$ is Gorenstein projective over $R$.

(2) By \Cref{essential}(2)(3), there exists a short exact sequence in $\mo(R)$:
$$
0\rightarrow K\rightarrow \Tr^G_R(\Omega^n_R(M))\rightarrow \Tr_A(M)\otimes_A \Ext^n_R(A,R)\rightarrow 0,
$$
where $\pd_R(K)\leq n-1$.
Combining this with that $i>0$, we get the second isomorphism below:
\begin{align*}
    \Ext^{n+i}_{R}(\Tr_R(\Omega^n_{R}(M)),R)&\cong \Ext^{n+i}_{R}(\Tr^G_R(\Omega^n_R(M)),R)\\
    & \cong \Ext^{n+i}_{R}(\Tr_A(M)\otimes_A\Ext^n_R(A,R),R)\\
    &=\h_{-(n+i)}\RHom_R(\Tr_A(M)\otimes_A\Ext^n_R(A,R),R))\\
    &=\h_{-(n+i)}\RHom_{A^{\rm op}}(\Tr_A(M),\RHom_{R}(\Ext^n_R(A,R),R))\\
    &=\h_{-(n+i)}\RHom_{A^{\rm op}}(\Tr_A(M),\Ext_{R}^n(\Ext^n_R(A,R),R)[-n])\\
   & =\h_{-(n+i)}\RHom_{A^{\rm op}}(\Tr_A(M),A[-n])\\
    &=\Ext^i_{A^{\rm op}}(\Tr_A(M),A)
\end{align*}
where the first one is by \ref{Gorenstein transpose}, the fourth one is by the adjunction $(-\otimes_A^{\rm L}\Ext^n_R(A,R),\RHom_R(\Ext^n_R(A,R),-))$ and $\Ext^n_R(A,R)$ is projective over $A$, the fifth one is because $\Ext_R^n(A,R)$ is projective over $A$ and $\Ext^i_A(A,R)=0$ for $i\neq n$, and the sixth one is due to \Cref{duality}.
\end{proof}
\begin{chunk}\label{MTT-result}
    Let $R$ be a commutative Noetherian ring. Recall that a finitely generated $R$-module $M$ is said to satisfy \emph{Serre's condition} $(S_n)$ if $\depth(M_\p)\geq \min\{n,\dim(R_\p)\}$ for all prime ideals $\p$ of $R$, where $\depth(M_\p)$ denotes the depth of $M_\p$ over $R_\p$.
     Recall also that  $R$ is said to satisfy the condition $(\widetilde{G}_n)$ provided that $R_\p$ is Gorenstein for all prime ideals $\p$ with $\depth(R_\p)\leq n$.
   
   For each $m>0$, Matsui, Takahashi, and Tsuchiya \cite[Theorem 1.4]{MTT} proved that the following conditions are equivalent:
    \begin{enumerate}
        \item $R$ satisfies $(\widetilde{G}_{m-1})$.
        \item For each finitely generated $R$-module $M$, one has:
        \begin{center}
            $M$ is $m$-torsionfree $\iff$ $M$ is $m$-th syzygy $\iff$ $M$ satisfies $(S_m)$.
        \end{center}
    \end{enumerate}
\end{chunk}

\begin{theorem}\label{AD}
 Let $\varphi\colon R\rightarrow A$ be a finite ring homomorphism. Assume $\Gpd_R(A)<\infty$, and  $
    \RHom_R(A,R)\cong P[-n] \text{ in }\D(A) \text{ for some } P\in\proj(A) \text{ and }n\geq 0.
    $ For each $M\in\mo(A)$ and $k>0$, we have{\rm :}
 \begin{enumerate}
\item Assume $n=0$. Then $M$ is $k$-torsionfree over $A$ if and only if $M$ is $k$-torsionfree over $R$.

\item   Assume $n\geq 1$ and $R$ is commutative. Consider the following conditions{\rm :}
\begin{enumerate}
    \item \label{111}
    $M$ is $k$-torsionfree over $A$

    \item\label{222}
    $\Tr_R\Omega^n_R\Tr_R\Omega^n_R(M)$ is $k$-torsionfree over $R$.

    \item \label{333}
    $\Omega^n_R(M)$ is $(k+n)$-torsionfree over $R$.
\end{enumerate}
Then $(\ref{333})\Longrightarrow (\ref{111})\iff (\ref{222})$. All these conditions are equivalent if, in addition, 
$R$ satisfies $(\widetilde{G}_{n-1})$.
    \end{enumerate}
    \end{theorem}

\begin{proof}
(1) This follows immediately from \Cref{changeofring-iso}(1).

(2) $(\ref{333})\Longrightarrow (\ref{111})$. Assume $\Omega^n_R(M)$ is $(k+n)$-torsionfree. By definition, we have $\Ext^j_R(\Tr_R\Omega^n_R(M),R)=0$ for all $1\leq j\leq k+n$. Combining this with \Cref{changeofring-iso}(2), $\Ext^i_{A^{\rm op}}(\Tr_A(M),A)=0$ for $1\leq i\leq k$, and hence $M$ is $k$-torsionfree over $A$.

$(\ref{111})\iff (\ref{222})$. By \Cref{changeofring-iso}(2), we have
$$
\Ext^i_{R}(\Tr_R\Tr_R\Omega^n_R\Tr_R\Omega^n_R(M),R)\cong \Ext^i_{A^{\rm op}}(\Tr_A(M),A)
$$
for each $i>0$.  In particular, $M$ is $k$-torsionfree over $A$ if and only if $\Tr_R\Omega^n_R\Tr_R\Omega^n_R(M)$ is $k$-torsionfree over $R$.

Assume, in addition, $R$ satisfies $(\widetilde{G}_{n-1})$. 
Combining with \ref{MTT-result}, we get that the module $\Omega^n_R(M)$ is $n$-torsionfree. Thus, $\Ext^i_R(\Tr_R\Omega^n_R(M),R)=0$ for $1\leq i\leq n$. Assume $(\ref{111})$ holds, then \Cref{changeofring-iso}(2) yields that $\Ext_R^{n+i}(\Tr_R(\Omega^n_R(M)),R)=0$ for $1\leq i\leq k$. Thus, we conclude that $\Ext^i_R(\Tr_R\Omega^n_R(M),R)=0$ for $1\leq i\leq n+k$. That is, $\Omega^n_R(M)$ is $(k+n)$-torsionfree over $R$.
\end{proof}
\begin{remark}\label{recover-Zhao}
 (1) When $\varphi\colon R \rightarrow A$ is further assumed to be a Frobenius extension, \Cref{AD}(1) was proved by Zhao \cite[Theorem A]{Zhao2024} using a different approach. Specifically, Zhao’s proof relies on a $\proj(R)$-approximation characterization of $k$-torsionfree modules, established by Auslander and Bridger \cite[Theorem 2.17]{Auslander-Bridger}.

 (2)  Let $\varphi\colon R\rightarrow A$ be a finite ring homomorphism, where $R$ is a commutative Noetherian ring and $A$ is an  $R$-algebra via $\varphi$. Assume $\varphi$ has ascent and descent of Gorenstein projective property in the sense of \cite{Liu-Ren}. For each $M\in\mo(A)$ and $k>0$, by \cite[Theorem 1.1]{Liu-Ren} and \Cref{AD}(1), $M$ is $k$-torsionfree over $A$ if and only if $M$ is $k$-torsionfree over $R$.
\end{remark}

Recall that a surjective ring homomorphism $\pi\colon R \rightarrow A$ between commutative Noetherian rings is called a \emph{complete intersection map} if the kernel of $\pi$ is generated by a regular sequence in $R$. When $\pi$ is a complete intersection map, $\pd_R(A)$ is finite.

Combining with \Cref{AD-Gorenstein dimension}, the following result follows from \Cref{AD}(2).
\begin{corollary}\label{AD-G-condition-torsionfree}
    Let $\varphi\colon R\rightarrow A$ be a finite ring homomorphism between commutative Noetherian rings, where $A$ is local. Assume $\varphi$ has ascent and descent of finite Gorenstein dimension property (cf. \ref{def-of-ad-gdim}) and $\Gpd_R(A)=n$. If $R$ satisfies $(\widetilde{G}_{n-1})$, then for each $M\in\mo(A)$ and $k>0$,  the module $M$ is $k$-torsionfree over $A$ if and only if $\Omega^n_R(M)$ is $(k+n)$-torsionfree over $R$.
\end{corollary}
\begin{corollary}\label{AD-Gorenstein rings}
    Let $\varphi\colon R\rightarrow A$ be a finite ring homomorphism between commutative Noetherian rings. Suppose that $n=\Gpd_R(A)$ and  $\varphi$ satisfies one of the following conditions{\rm :}
    \begin{enumerate}
        \item $R$ and $A$ are commutative Gorenstein rings, and $A$ is also local.

        \item  $\varphi\colon R\rightarrow A$ is a complete intersection map, and $R$ satisfies $(\widetilde{G}_{n-1})$.
    \end{enumerate}
  Then for each $M\in\mo(A)$ and $k>0$, the module
$M$ is $k$-torsionfree over $A$ if and only if $\Omega^n_R(M)$ is $(k+n)$-torsionfree over $R$.
\end{corollary}
\begin{proof}
  (1) Since $R$ and $A$ are Gorenstein, any finitely generated module over $R$ or $A$ has finite Gorenstein dimension; see \cite[Corollary 2]{Goto}. In particular, $\Gpd_R(A)<\infty$ and $\varphi$ has ascent and descent of finite Gorenstein dimension property. The desired result now follows from \Cref{AD-G-condition-torsionfree}.

  (2) Note that $n=\pd_R(A)$ as $\varphi$ is a complete intersection map. Moreover, $\RHom_R(A,R)\cong A[-n]$ in $\D(A)$; see \cite[Proposition 1.6.10 and Corollary
1.6.14]{BH}. The desired result follows from \Cref{AD}(2).
\end{proof}
\begin{remark}
    Keep the assumption as \Cref{AD-Gorenstein rings}(1). Combining with \ref{MTT-result}, it follows from \Cref{AD-Gorenstein rings}(1) that, for each $M\in\mo(A)$ and $k>0$, the module $M$ is a $k$-th syzygy in $\mo(A)$ if and only if $\Omega^n_R(M)$ is a $(k+n)$-th syzygy in $\mo(R)$, where $n=\Gpd_R(A)$.
\end{remark}

The following example satisfies the assumption of \Cref{AD-Gorenstein rings}(2), but does not satisfy the assumption of \Cref{AD-Gorenstein rings}(1).
\begin{example}
    Let $R=k[t^3,t^4,t^5]$ be a numerical semigroup ring, where $k$ is a field. For each nonzero, non-invertible element $r\in R$, consider the canonical surjection
    $$
    \pi\colon R\rightarrow A\colonequals R/(r).
    $$
Since $R$ is a one-dimensional domain, we get that $\varphi$ is a complete intersection map and $R_{(0)}$ is a field. It follows that $R$ satisfies $(\widetilde{G}_0)$. Thus, for each $M\in\mo(A)$ and $k>0$, \Cref{AD-Gorenstein rings}(2) yields that the module $M$ is $k$-torsionfree over $A$ if and only if $\Omega^1_R(M)$ is $(k+1)$-torsionfree over $R$. 
Note that the semigroup $\langle 3,4,5\rangle$ is not symmetric in the sense of \cite[Theorem 4.4.8]{BH}, and hence $R$ is not Gorenstein. 
\end{example}
\section{Extension closedness of $k$-torsionfree modules}\label{Section 5}
In this section, we investigate the extension closedness of the category of \( k \)-torsionfree modules under change of rings. The main result is \Cref{T2} from the introduction; see \Cref{AD of extension closedness}. As an application, we establish an ascent and descent result concerning quasi \( k \)-Gorensteinness; see \Cref{AD-quasi-Gorenstein}. This provides an affirmative answer to a question posed by Zhao regarding quasi \( k \)-Gorensteinness in the case of Noetherian algebras; see \Cref{affirmative}.

\begin{lemma}\label{transfer transpose}
    Let $R\rightarrow A$ be a ring homomorphism and $X$ be a finitely generated module over $R$. Assume $T$ is an $A\mbox{-}R$-bimodule and $T$ is finitely generated projective over $A$. Then there is an isomorphism in $\mo(A^{\rm op}){\rm:}$
    $$
     \Tr_R(X)\otimes_R \Hom_A(T,A)\cong \Tr_A(T\otimes_R X).
    $$
\end{lemma}
\begin{proof}
    Choose a projective resolution of $X$ over $R$:
$$
P_1\xrightarrow {d_1} P_0\xrightarrow{d_0} X\rightarrow 0,
$$
where $P_0, P_1\in\proj(R)$.
Note that $T$ is finitely generated projective over $A$. This yields that $T\otimes_R P_i$ is finitely generated projective over $A$ for $i=0,1$. By applying the functor $T\otimes_R-$, we get a projective resolution of $T\otimes_R X$ in $\mo(A)$:
$$
T\otimes_RP_1\xrightarrow {T\otimes_R d_1} T\otimes_RP_0\xrightarrow{T\otimes_R d_0} T\otimes_RX\rightarrow 0.
$$
The above short exact sequences induce the following commutative diagrams in $\mo(A^{\rm op})$ with exact rows:
$$\xymatrix{
\Hom_A(T\otimes_R P_0,A)\ar[r]\ar[d]^-\cong&  \Hom_A(T\otimes_R P_1,A)\ar[r]\ar[d]^-\cong & \Tr_A(T\otimes_R X)\ar[r]\ar[d]& 0\\
\Hom_R(P_0,\Hom_A(T,A))\ar[r] & \Hom_R(P_1,\Hom_A(T,A))\ar[r] & C\ar[r]& 0
,}
$$
where the two isomorphisms are due to the adjunction $(T\otimes_R-, \Hom_A(T,-))$. By the Five lemma, we have $\Tr_A(A\otimes_RX)\cong C$ in $\mo(A^{\rm op})$. By \cite[Propositon 4.3]{Takahashi2013}, $C\cong \Tr_R(X)\otimes_R \Hom_A(T,A)$ in $\mo(A^{\rm op})$, and hence 
$\Tr_R(X)\otimes_R \Hom_S(T,A)\cong \Tr_A(T\otimes_R X)$ in $\mo(A^{\rm op})$. 
\end{proof}

\begin{lemma}\label{change-tensor-torsionfree}
    Let $\varphi\colon R\rightarrow A$ be a ring homomorphism. Assume $T$ is an $A\mbox{-}R$-bimodule such that $T$ is finitely generated projective over $A$, $T$ is flat over $R^{\rm op}$, and   $\Hom_A(T,A)$ is flat over $R$. Let $X$ be a finitely generated $R$-module and $k>0$. If $X$ is $k$-torsionfree over $R$, then $T\otimes_R X$ is $k$-torsionfree over $A$. The converse holds if, in addition, $T$ is faithfully flat over $R^{\rm op}$. 
\end{lemma}

\begin{proof}
    By assumption, $T$ is a finitely generated projective module over $A$. It follows that $\Hom_A(T,A)$ is a finitely generated projective module over $A^{\rm op}$ and there exists an isomorphism in $\mo(R^{\rm op})$:
    \begin{equation}\label{reflexive of T}
        T\xrightarrow \cong \Hom_{A^{\rm op}}(\Hom_A(T,A),A).
    \end{equation}
Combining this with that $\Hom_A(T,A)$ is flat over $R$, 
by \Cref{Ext-isomorphism}, for each $i\geq 0$, we obtain the second isomorphism below:
    \begin{align*}
        \Ext^i_{A^{\rm op}}(\Tr_A(T\otimes_R X),A)& \cong \Ext^i_{A^{\rm op}}( \Tr_R(X)\otimes_R \Hom_A(T,A),A)\\
        & \cong\Ext^i_{R^{\rm op}}(\Tr_R(X),T)\\
        & \cong T\otimes_R \Ext^i_{R^{\rm op}}(\Tr_R(X), R),
    \end{align*}
    where the first isomorphism is by \Cref{transfer transpose}, and the last uses that $T$ is flat over $R^{\rm op}$ (see \cite[A.4.23]{Christensen-book}). Thus, if $X$ is $k$-torsionfree over $R$, then $T\otimes_R X$ is $k$-torsionfree over $A$. If furthermore $T$ is faithfully flat over $R^{\rm op}$, then the above isomorphisms yields that $X$ is $k$-torsionfree over $R$ if $T\otimes_R X$ is $k$-torsionfree over $A$.
\end{proof}

Motivated by \Cref{change-tensor-torsionfree}, we introduced the following definitions.
\begin{definition}\label{def of quasi-faithfully flat extension}
    Let $\varphi\colon R\rightarrow A$ be a ring homomorphism. We say $\varphi$ is a \emph{quasi-faithfully flat extension} if there exists an $A\mbox{-}R$-bimodule $T$ such that:
    \begin{enumerate}
        \item $T$ is a finitely generated projective over $A$;
\item $T$ is faithfully flat over $R^{\rm op}$;
        \item $\Hom_A(T,A)$ is flat over $R$. 
    \end{enumerate}  
The homomorphism $\varphi$ is said to be a \emph{faithfully flat extension} if $A$ is faithfully flat over $R^{\rm op}$ and $A$ is flat over $R$.
\end{definition}
\begin{remark}
(1) A faithfully flat extension is a quasi-faithfully flat extension. The converse of this is not true in general. Let $k$ be a field and $G$ be a finite group. By Example \ref{example-weak-ff}, the canonical map $kG\rightarrow k;~g\mapsto 1$ is a quasi-faithfully flat extension. But this is not a faithfully flat extension if $|G|\geq 2$; 
this is because any faithfully flat extension is injective; see \cite[Theorem 4.7.4]{Lam1999}.


(2) We don't know that whether the definition of (quasi-)faithfully flat extension is left-right symmetric. Specifically, given a ring homomorphism $\varphi\colon R\rightarrow A$, whether the homomorphism $\varphi$ is a (quasi-)faithfully flat extension is equivalent to that $\varphi^{\rm op}\colon R^{\rm op}\rightarrow A^{\rm op}$ is a (quasi-)faithfully flat extension. If one requires that $\Hom_A(T,A)$ is faithfully flat over $R$ in \Cref{def of quasi-faithfully flat extension}, then the definition will be left and right symmetric.
\end{remark}

Besides the faithfully flat extensions, the following gives another class of ring homomorphisms that are quasi-faithfully flat extensions. 
\begin{example}\label{example-weak-ff}
 Let $k$ be a field and $R$ be a Frobenius $k$-algebra; cf. \Cref{def of Frobenius}. Then any $k$-algebra homomorphism $\varphi\colon R\rightarrow A$ is a quasi-faithfully flat extension. 

Indeed, take $T=A\otimes_k R$. Since $k$ is a field, $T$ is free over $A$ and $R^{\rm op}$. In particular, $T$ is faithfully flat over $R^{\rm op}$. Also, we have an isomorphism over both $R$ and $A^{\rm op}$:
\begin{align*}
    \Hom_A(A\otimes_k R,A)& \cong \Hom_k(R,A)\\
    & \cong \Hom_k(R,k)\otimes_k A\\
    & \cong R\otimes_k A,
\end{align*}
where the first isomorphism is due to the adjunction, the second one is due to that $R$ is a free module over $k$ with finite rank, and the last one uses that $R$ is a Frobenius $k$-algebra. It follows that $\Hom_A(T,A)$ is free over $R$, and hence faithfully flat over $R$. Hence, $\varphi$ is a quasi-faithfully flat extension. 
\end{example}

\begin{chunk}
    Let $\mathcal{X}$ be a full subcategory of $\mathsf{mod}(R)$. The category $\mathcal{X}$ is said to be \emph{extension closed} if it satisfies the following property: for each short exact sequence
\[
0 \rightarrow X \rightarrow Y \rightarrow Z \rightarrow 0
\]
in $\mathsf{mod}(R)$, if $X, Z \in \mathcal{X}$, then $Y \in \mathcal{X}$ as well.
\end{chunk}
\begin{proposition}\label{descent of extension closedness}
    Let $\varphi\colon R\rightarrow A$ be a quasi-faithfully flat extension.  For each $k>0$, if $\TF^k(A)$ is extension closed, then so is $\TF^k(R)$. 
\end{proposition}
\begin{proof}
 By the assumption, there exists an $A\mbox{-}R$-bimodule such that $T$ is finitely generated projective over $A$ and faithfully flat over $R^{\rm op}$, and $\Hom_A(T,A)$ is flat over $R$. Consider the short exact sequence in $\mo(R)$:
    $$
    0\rightarrow X\rightarrow Y\rightarrow Z\rightarrow 0
    $$
    with $X,Z\in\TF^k(R)$. By \Cref{change-tensor-torsionfree}, both $T\otimes_R X$ and $T\otimes_R Z$ are in $\TF^k(A)$. Since $T$ is flat over $R^{\rm op}$, the above short exact sequence induces a short exact sequence in $\mo(A)$:
    $$
    0\rightarrow T\otimes_R X\rightarrow T\otimes_R Y\rightarrow T\otimes_R Z\rightarrow 0. 
    $$
    Since $\TF^k(A)$ is extension closed, we get that $T\otimes_R Y\in \TF^k(A)$. Again by \Cref{change-tensor-torsionfree}, we have $Y\in \TF^k(R)$, and hence $\TF^k(R)$ is extension closed. 
\end{proof}

\begin{proposition}\label{ascent of extension closedness}
      Let $\varphi\colon R\rightarrow A$ be a finite ring homomorphism.  Assume that $\Gpd_R(A)<\infty$, and  $
    \RHom_R(A,R)\cong P[-n] \text{ in }\D(A) \text{ for some } P\in\proj(A) \text{ and }n\geq 0$. Assume either $n=0$,  or that $R$ is commutative which satisfies $(\widetilde{G}_{n-1})$.  For each $k>0$,  if $\TF^{k+n}(R)$ is extension closed, then so is $\TF^k(A)$.
\end{proposition}
   
\begin{proof}
Assume $\TF^{k+n}(R)$ is extension closed. For each short exact sequence in $\mo(A)$:
    $$
    0\rightarrow M_1\rightarrow M_2\rightarrow M_3\rightarrow 0
    $$
    with $M_1,M_3\in \TF^k(A)$. Restriction scalars along $\varphi\colon R\rightarrow A$ and combining with the Horseshoe lemma, we have a short exact sequence in $\mo(R)$:
    $$
    0\rightarrow \Omega^n_R(M_1)\rightarrow \Omega^n_R(M_2)\rightarrow \Omega^n_R(M_3)\rightarrow 0.
    $$
    By \Cref{AD}, both $\Omega^n_R(M_1)$ and $\Omega^n_R(M_3)$ are in $\TF^{k+n}(R)$. It follows that $\Omega^n_R(M_2)\in \TF^{k+n}(R)$ as $\TF^{k+n}(R)$ is extension closed. Again by \Cref{AD}, we have $M_2\in \TF^k(A)$, and hence $\TF^k(A)$ is extension closed. 
\end{proof}
\begin{theorem}\label{AD of extension closedness}
     Let $\varphi\colon R\rightarrow A$ be a finite ring homomorphism, where $\varphi$ is a quasi-faithfully flat extension. Assume $\Gpd_R(A)=0$ and $\Hom_R(A,R)$ is projective over $A$. For each $k>0$,  the category $\TF^{k}(R)$ is extension closed if and only if so is $\TF^k(A)$.
\end{theorem}
\begin{proof}
    This follows immediately from \Cref{descent of extension closedness} and \Cref{ascent of extension closedness}.
\end{proof}

By \Cref{example-weak-ff} and \Cref{AD of extension closedness}, we get the following result.
\begin{corollary}
     Let $k$ be a field and $R$ be a Frobenius $k$-algebra and $\varphi\colon R\rightarrow A$ be a finite $k$-algebra homomorphism. Assume $\Gpd_R(A)=0$, $\Hom_R(A,R)$ is projective over $A$. For each $k>0$,  the category $\TF^{k}(R)$ is extension closed if and only if so is $\TF^k(A)$.
\end{corollary}
The following result is an immediate consequence of \Cref{AD of extension closedness}. The first part was previously established by Zhao \cite[Proposition 4.3]{Zhao2024}.
\begin{corollary}\label{ext-closed-torsionfree}
     Let $R\rightarrow A$ be a Frobenius extension and $n>0$. If $\TF^n(R)$ is extension closed, then so is $\TF^n(A)$. The converse holds if furthermore $A$ is faithfully flat over $R^{\rm op}$.
\end{corollary}

   The next example satisfies the assumption of \Cref{AD of extension closedness}, but $A$ is not projective over $R$. In particular, it is not a Frobenius extension.
\begin{example}\label{Example-not proj}\label{Example-not proj}
    Let $A$ be a commutative Noetherian ring and consider the canonical surjection:
    $$
    \pi\colon R=A\llbracket x\rrbracket /(x^2)\twoheadrightarrow A.
    $$
Note that $A$ is Gorenstein projective over $R$ and $\Hom_R(A,R)\cong A$ as $A$-modules; see for example \cite[Example 4.9]{Liu-Ren}. However, $A$ is not projective over $R$. Consider $T=R$. This is an $A\mbox{-}R$-bimodule. The module $T$ is free over $A$ and faithfully flat over $R$. Moreover, $\Hom_A(R,A)\cong R$ as $R$-modules; see for example \cite[Lemma 3.1]{Ren-SCM}. Thus, $\pi$ is a quasi-faithfully flat extension. 
\end{example}

\begin{chunk}\label{def quasi Gorenstein}
    Let $R$ be a Noetherian ring. Assume that \( k \) is a positive integer and 
\[
0 \rightarrow R \rightarrow I^0 \rightarrow I^1 \rightarrow \cdots \rightarrow I^i \longrightarrow \cdots
\]
is a minimal injective resolution of the module \( R\) over $R$.

(1) The ring \( R\) is called a \emph{left quasi \( k \)-Gorenstein}, as defined by Huang \cite[Definition 2]{Huang-science}, if 
$
\operatorname{fd}_R(I^i) \leq i + 1  \text{ for all } 0 \leq i \leq k - 1;
$
here $\operatorname{fd}_R(-)$ represents the flat dimension over $R$.
Note that the notion of quasi \(k\)-Gorensteinness is not left-right symmetric in general; see \cite[Example 2]{Huang-science}. However, \Cref{symmetric case} provides sufficient conditions under which quasi \(k\)-Gorensteinness becomes symmetric. The ring \(R\) is called \emph{quasi \(k\)-Gorenstein} if both \(R\) and \(R^{\mathrm{op}}\) are left quasi \(k\)-Gorenstein.

(2) Recall from \cite{FGR} that the ring \( R\) is called  a \emph{\( k \)-Gorenstein}  provided that
$
\operatorname{fd}_R(I^i) \leq i \quad \text{for all } 0 \leq i \leq k - 1. 
$
The notion of $k$-Gorenstein rings is left-right symmetric; see \cite[Theorem 3.7]{FGR}. That is, $R$ is $k$-Gorenstein if and only if $R^{\rm op}$ is $k$-Gorenstein.

\end{chunk}

\begin{remark}
Let $\Lambda$ be an Artin algebra. Auslander and Reiten \cite{Auslander-Reiten1994} posed the following question: if $\Lambda$ is $k$-Gorenstein for all $k > 0$, is $\Lambda$ necessarily Iwanaga-Gorenstein? As noted in loc. cit., an affirmative answer would also imply the Nakayama conjecture, which asserts that an Artin algebra with infinite dominant dimension is self-injective. In this context, Iyama and Zhang \cite[Corollary 1.3]{IZ} provided equivalent conditions under which a $k$-Gorenstein Artin algebra for all $k > 0$ is Iwanaga-Gorenstein.
\end{remark}

\begin{chunk}\label{ARH}
Let $R$ be a Noetherian algebra.
 For a positive integer \( k \), the following conditions are equivalent: 

\begin{enumerate}
  \item \( \Omega^i_R(\mo(R)) \) is extension closed for \( 1 \leq i \leq k \);
  \item If 
 $
  0 \to R \to I^0 \to I^1 \to \cdots \to I^i \to \cdots
$
  is a minimal injective resolution of the module \( R \) over $R$, then 
  $
  \operatorname{fd}_R(I^i) \leq i + 1 \quad \text{for } 0 \leq i \leq k - 1;
 $
  \item \( \TF^i(R) \) is extension closed for \( 1 \leq i \leq k \).
\end{enumerate}
The equivalence $(1)\iff (2)$ was established by Auslander and Reiten \cite[Theorem 1.7 and Proposition 2.2]{Auslander-Reiten} for Noetherian rings, and the equivalence $(2)\iff (3)$ was proved by Huang \cite[Theorem 3.3]{Huang1999} for Noetherian algebras. Note that the assumption that $R$ is a Noetherian algebra is needed in the proof of the loc. cit.
\end{chunk}
Combining with \Cref{ARH}, the following result is an immediate consequence of \Cref{AD of extension closedness}.
\begin{corollary}\label{AD-quasi-Gorenstein}
     Let $R$ and $A$ be two Noetherian algebras, and let $\varphi\colon R\rightarrow A$ be a finite ring homomorphism such that $\varphi$ is a quasi-faithfully flat extension. Assume $\Gpd_R(A)=0$ and $\Hom_R(A,R)$ is projective over $A$. For each $k>0$, the ring $R$ is left quasi $k$-Gorenstein if and only if so is $A$. 
\end{corollary}
\begin{remark}\label{symmetric case}
Let \(\varphi\colon R \rightarrow A\) be a finite ring homomorphism, where \(R\) is commutative and \(A\) is an \(R\)-algebra via \(\varphi\). Assume that \(\varphi\) is a quasi-faithfully flat extension,  \(\operatorname{Gpd}_R(A) = 0\), and  \(\operatorname{Hom}_R(A, R)\) is projective over $A$ on both sides.
Since \(R\) is commutative, the notion of quasi \(k\)-Gorensteinness for \(R\) is left-right symmetric. Combining with \ref{AD-quasi-Gorenstein}, it follows that for each \(k > 0\), \(A\) is left quasi \(k\)-Gorenstein if and only if it is right quasi \(k\)-Gorenstein. In other words, quasi \(k\)-Gorensteinness for \(A\) is left-right symmetric under these assumptions. It is worth mentioning that, in general, quasi \(k\)-Gorensteinness is not left-right symmetric; see \cref{def quasi Gorenstein}.
\end{remark}

\begin{chunk}
  (1)   Let $R$ and $A$ be two Noetherian algebras, and $\varphi\colon R\rightarrow A$ be a Frobenius extension. If $R$ is left quasi $k$-Gorenstein, then Zhao \cite[Theorem 4.1]{Zhao2024} observed that $A$ is also left quasi $k$-Gorenstein. It is also proved in \cite[Proposition 5.2]{Zhao2024} that the same result holds for $k$-Gorenstein. 
     
  (2) Recall that a module $X$ in $\mo(R)$ is called \emph{generator} if, for any $Y\in \mo(R)$, there exists a surjective $R$-linear map $X^n\rightarrow Y$ for some $n\geq 0$.  For any Frobenius extension $R\rightarrow A$, where $A$ is a generator in both $\mo(R)$ and $\mo(R^{\rm op})$, Chen and Ren \cite[Example 3.8]{Chen-Ren} observed that $R$ is $k$-Iwanaga Gorenstein (c.f. \ref{def of I-Gorenstein}) if and only if $A$ is $k$-Iwanaga Gorenstein. 

   (3) Based on the above, Zhao \cite[Section 5]{Zhao2024} raised the following questions:

   \textbf{Questions:} Let $\varphi\colon R\rightarrow A$ be a Frobenius extension, where $A$ is a generator in both $\mo(R)$ and $\mo(R^{\rm op})$.
   \begin{itemize}
\item For each $k>0$, is $R$ (quasi) $k$-Gorenstein equivalent to that $A$ is (quasi) $k$-Gorenstien?

       \item Is $R$ $k$-Gorenstein for all $k>0$ equivalent to that $A$ is $k$-Gorenstein for all $k>0$?
 \end{itemize}

\end{chunk}

 \Cref{AD-quasi-Gorenstein} yields the following result, which provides an affirmative answer to Zhao's question regarding quasi $k$-Gorenstein when $R$ and $A$ are Noetherian algebras. Note that a projective generator is faithfully flat. 
\begin{corollary}\label{affirmative}
    Let  $R$ and $A$ be Noetherian algebras, and let $\varphi\colon R\rightarrow A$ be a Frobenius extension. Assume $A$ is faithfully flat over $R^{\rm op}$.  For each $k>0$, the ring $R$ is left quasi $k$-Gorenstein if and only if so is $A$. 
\end{corollary}
\begin{remark}
    Let $\varphi\colon R\rightarrow A$ be a Frobenius extension, where $R$ is a commutative Noetherian ring and the image of $\varphi$ is in the center of $A$. Then \Cref{affirmative} yields that Zhao's question regarding quasi $k$-Gorenstein is true in this case.
\end{remark}
\begin{proposition}\label{main}
    Let $\varphi\colon R\rightarrow A$ be a Frobenius extension. Assume $R$ is commutative and the image of $\varphi$ is in the center of $A$, and $A$ is faithfully flat over $R^{\rm op}$.  For each $k>0$, the ring $R$ is $k$-Gorenstein if and only if so is $A$. 
\end{proposition}
\begin{proof}
The forward direction is due to \cite[Proposition 5.2]{Zhao2024}. For the backward direction, assume $A$ is $k$-Gorenstein. Choose a minimal injective resolution of $R$
  \[
  0 \to R \to I^0 \to I^1 \to \cdots \to I^i \to \cdots
  \]
over $R$. By the proof of \cite[Theorem 1.2]{Eisenbud1970}, the above minimal injective resolution induces a minimal injective resolution
 \[
  0 \to \Hom_R(A,R) \to \Hom_R(A,I^0) \to \Hom_R(A,I^1) \to \cdots \to \Hom_R(A,I^i) \to \cdots.
  \]
  Since $A$ is Frobenius extension over $R$, we have $A\cong \Hom_R(A,R)$  and $\Hom_R(A,I_i)\cong A\otimes_R I_i$ in $\mo(A)$. It follows that we have a minimal injective resolution
$$
0\to A\to A\otimes_R I^0\to A\otimes_R I^1\to \cdots \to A\otimes_R I^i\to \cdots
$$
over $A$.
For each $R$-module $X$ and $n\geq 0$, we have
\begin{align*}
    \Tor_n^A(X\otimes_R A, A\otimes_R I^i)& \cong \Tor_n^R(X\otimes_R A, I^i)\\
    & \cong\Tor_n^R(A\otimes_R X,I^i)\\
    & \cong A\otimes_R \Tor_n^R(X,I^i),
\end{align*}
where the second isomorphism is because $R$ is commutative and the last one is due to that $A$ is flat over $R$. Combining the above isomorphism with that $A$ is faithfully flat over $R^{\rm op}$, we conclude that $\operatorname{fd}_R(I^i)\leq \operatorname{fd}_A(A\otimes_R I^i)$. Since $A$ is $k$-Gorenstein, we get that $R$ is also $k$-Gorenstein.
\end{proof}
\begin{remark}
    (1) \Cref{main} supports Zhao's question regarding $k$-Gorenstein.

    (2) In \cite[Corollary 4.7]{GHZ}, Gu, Huang, and Zhao proved the following: for a Forbenius extension $\varphi\colon R\rightarrow A$, if $\varphi$ is a split monomorphism over $R^{\rm op}$, then, for each $k>0$, the ring $R$ is $k$-Gorenstein if and only if $A$ is $k$-Gorenstein. This also supports Zhao's question regarding $k$-Gorenstein. 
\end{remark}


\section{Finite representation type}\label{Section 6}

In this section, we investigate the finite representation type of $k$-torsionfree modules over a Frobenius extension; see \Cref{asent and descent of finite type of TF}. The result can be applied to the skew group rings; see \Cref{skew-application}.

\begin{chunk}
   (1) A ring homomorphism $R\rightarrow A$ is said to be a \emph{separable} if the map
    $$
   A\otimes_R A\rightarrow A;~ a\otimes b\mapsto ab
    $$
    is a split epimorphism as $A\mbox{-}A$-bimodules.

  (2)  For a ring homomorphism $R\rightarrow A$, it is said to be \emph{split} if $A\cong R\oplus M$ as $R\mbox{-}R$-bimodules for some $R\mbox{-}R$-bimodule $M$. 
\end{chunk}
\begin{example}\label{skew}
    Let $\Lambda$ be a ring and $G$ be a finite group acting on $\Lambda$. Then the skew group ring $\Lambda G$ is a split Frobenius extension over $\Lambda$; see \cite[Lemma 4.5 and Proposition 2.13]{ARS}. Note that the ordinary group rings are special cases of skew group rings.

If, in addition, the order of $G$, denoted by $|G|$, is invertible in $\Lambda$, then $\Lambda G$ is also a separable extension over $R$. Indeed, choose $e=|G|^{-1}\sum_{g\in G}g\otimes g^{-1}\in\Lambda G\otimes_\Lambda\Lambda G$, then for each $a\in \Lambda G$, $ae=ea$, and hence \cite[Lemma 2.9]{Ren-SCM} implies that $\Lambda G$ is separable over $\Lambda$.
\end{example}
\begin{chunk}\label{split-direct summand}
(1) Assume $R\rightarrow A$ is a separable extension. For any module $M$ over $A$, note that $M$ is a direct summand of $A\otimes_R M$ as an $A$-module.

  (2) Assume $R\rightarrow A$ is a split extension. For any module $M$ over $R$, note that $M$ is a direct summand of $A\otimes_R M$; see \cite{Ren-Corrigendum}. 
\end{chunk}

\begin{chunk}
  Following \cite{LW-book}, we say KRS holds for $\mo(R)$ if the Krull-Remak-Schmidt theorem holds for finitely generated left $R$-modules; see \cite[Chapter 1]{LW-book}.
  Note that KRS holds for finitely generated left modules over the following class of rings:
\begin{enumerate}
    \item Left Artinian rings; see \cite[Corollary 19.23]{Lam}.
    \item Commutative Henselian local rings \cite[Theorem 1.8]{LW-book}. For example, any complete local ring is Henselian.
\end{enumerate}
\end{chunk}
Let  $\mathcal X$ be a full subcategory of $\mo(R)$. $\mathcal X$ is said to have \emph{finite representation type} if only finitely many isomorphism classes of indecomposable modules exist in $\mathcal{X}$.
\begin{lemma}\label{add}
   Assume KRS holds for finitely generated modules over $R$. Assume $\mathcal X$ is a full subcategory of $\mo(R)$ that is closed under finite direct sums and direct summands. Then
    $\mathcal X$ has finite representation type if and only if $\mathcal X=\add_R(M)$ for some $M\in\mathcal X$. 
\end{lemma}
\begin{proof}
This follows the same argument as in the proof of \cite[Lemma 3.11]{DKLO}. 
\end{proof}
\begin{proposition}\label{asent and descent of finite type of TF}
     Let $\varphi\colon R\rightarrow A$ be a Frobenius extension. Assume KRS holds for $\mo(R)$ and $\mo(A)$.  Then{\rm :}
     
     (1) Assume $\varphi$ is separable. If $\TF^n(R)$ has finite representation type, then so does $\TF^n(A)$.

(2) Assume $\varphi$ is split. If $\TF^n(A)$ has finite representation type, then so does $\TF^n(R)$.
\end{proposition}
\begin{proof}
   Note that the category of $n$-torsionfree modules is closed under finite direct sums and direct summands. 

    (1) By assumption and Lemma \ref{add}, $\TF^n(R)=\add_R(M)$ for some $M\in\TF^n(R)$. It follows from \Cref{change-tensor-torsionfree} that $A\otimes_R M$ is in $\TF^n(A)$. Combining again with \Cref{add}, it suffices to show $\TF^n(A)=\add_R(A\otimes_R M)$. As $A\otimes_R M\in\TF^n(A)$ and $\TF^n(A)$ is closed under direct summands, it remains to show that $\TF^n(A)\subseteq \add_A(A\otimes_R M)$. For each $X\in \TF^n(A)$, it follows from \Cref{recover-Zhao} that $X\in \TF^n(R)=\add_R(M)$, and hence $A\otimes_R X\in \add_A(A\otimes_R M)$ by \Cref{change-tensor-torsionfree}. Since $\varphi$ is separable, $X$ is a direct summand of $A\otimes_R X$ as an $A$-module; see \ref{split-direct summand}. Therefore, $X\in \add_A(A\otimes_R M)$, and hence $\TF^n(R)\subseteq \add_A(A\otimes_RM)$. 

    (2) By assumption and Lemma \ref{add}, $\TF^n(A)=\add_A(N)$ for some $N\in\TF^n(A)$.  Since $\varphi$ is Frobenius, it follows that $N\in\TF^n(R)$; see \Cref{recover-Zhao}. It suffices to prove $\TF^n(R)=\add_R(N)$ by \Cref{add}. For each $X\in \TF^n(R)$, it follows from  \Cref{change-tensor-torsionfree} that $A\otimes_R X$ is in $\TF^n(A)$. Thus, $A\otimes_R X\in \add_A(N)$. Rectriction scalars along $\varphi$, we have $A\otimes_R X\in\add_R(N)$. Combining with \ref{split-direct summand}, we conclude that $X\in\add_R(N)$, and hence $\TF^n(R)=\add_R(N)$. 
\end{proof}

\begin{remark}\label{CM-finite result}
  Let $\varphi\colon R\rightarrow A$ be a Frobenius extension. 

  (1)
 For each $M\in\mo(A)$, the module $M$ is Gorenstein projective over $A$ if and only if it is Gorenstein projective over $R$; see \cite{Chen}.  Combining this with  \cite[Lemma 2.3]{Ren-SCM}, the same argument of \Cref{asent and descent of finite type of TF} yields the following: Assume, in addition, $R$ is separable split, and KRS holds for $\mo(R)$ and $\mo(A)$. Then:
$R$ is CM-finite if and only if $A$ is CM-finite;
recall that a ring $R$ is called CM-finite if $\Gproj(R)$ has finite representation type.

(2) It is worth mentioning that there is a result contained in \cite[Theorem B]{Zhao2019}: Assume, in addition, $R$ is a commutative Artinian ring, $A$ is an $R$-algebra via $\varphi$, and $\varphi$ is separable. Then $R$ is CM-finite if and only if $A$ is CM-finite. However, \Cref{counterexample} shows that this result and the result concerning CM-free in \cite[Theorem B]{Zhao2019} need not hold without the split assumption.
\end{remark}

\begin{example}\label{counterexample}
    Consider the canonical projection
    $$
   \pi\colon R\colonequals k\times k\llbracket x,y\rrbracket/(x^2,y^2)\rightarrow A\colonequals k. 
    $$
    This is a Frobenius extension; see  \cite[Remark 4.4]{Zhao2024} or \Cref{AD-Gorenstein dimension}.  Moreover, $\pi$ is a separable Frobenius extension. This is because, for each $r\in k$, we have $$r\otimes_R 1=1\cdot (r,0)\otimes_R 1=1\otimes_R (r,0)\cdot 1=1\otimes_R r$$
    in $A\otimes_R A$, and hence there is a homomorphism
    $$
   \iota\colon A\rightarrow A\otimes_R A;~r\mapsto r\otimes_R 1 
    $$
    of $A\mbox{-}A$-bimodules which satisfies $\epsilon\circ \iota =\id_A$, where $\epsilon\colon A\otimes_R A\rightarrow A; r_1\otimes r_2\mapsto r_1r_2$. 
    
Note that $A = k$ is CM-free, and thus CM-finite.  However, $R$ is not CM-finite. Indeed, there is a surjective map
$
R\rightarrow k\llbracket x,y\rrbracket/(x,y)^2; ~(a,b)\mapsto \overline{b}.
$
It is known that $k\llbracket x,y\rrbracket/(x,y)^2$ is of infinite representation type. This yields that $R$ is not CM-finite as $R$ is self-injective. 
\end{example}

Combining with Example \ref{skew}, the following result is an immediate consequence of \Cref{asent and descent of finite type of TF} and  \Cref{CM-finite result}.
\begin{corollary}\label{skew-application}
    Let $\Lambda$ be a left Artinian ring and $G$ be a finite group acting on $\Lambda$. Assume $|G|$ is invertible in $\Lambda$. Then{\rm :}
\begin{enumerate}
    \item For each $k>0$, $\TF^k(\Lambda G)$ has finite representation type if and only if so does 
    $\TF^k(\Lambda)$.

    \item The skew group ring $\Lambda G$ is CM-finite if and only if so is $\Lambda$.
\end{enumerate}
\end{corollary}

\bibliographystyle{amsplain}
\bibliography{ref}

@article{Huang-Huang,
 author = {Huang, Chonghui and Huang, Zhaoyong},
 title = {Gorenstein syzygy modules},
 fjournal = {Journal of Algebra},
 journal = {J. Algebra},
 issn = {0021-8693},
 volume = {324},
 number = {12},
 pages = {3408--3419},
 year = {2010},
 doi = {10.1016/j.jalgebra.2010.10.010},
 zbMATH = {5860497},
 Zbl = {1216.16002}
}

@article{Zhao2024,
 author = {Zhao, Zhibing},
 title = {{{\(k\)}}-torsionfree modules and {Frobenius} extensions},
 fjournal = {Journal of Algebra},
 journal = {J. Algebra},
 issn = {0021-8693},
 volume = {646},
 pages = {49--65},
 year = {2024},
 doi = {10.1016/j.jalgebra.2024.01.028},
 zbMATH = {7818591},
 Zbl = {1548.16004}
}

@article{Auslander-Reiten,
 author = {Auslander, Maurice and Reiten, Idun},
 title = {Syzygy modules for {Noetherian} rings},
 fjournal = {Journal of Algebra},
 journal = {J. Algebra},
 issn = {0021-8693},
 volume = {183},
 number = {1},
 pages = {167--185},
 year = {1996},
 doi = {10.1006/jabr.1996.0212},
 zbMATH = {923677},
 Zbl = {0857.16006}
}

@article{Auslander-Reiten1994,
 author = {Auslander, Maurice and Reiten, Idun},
 title = {{{\(k\)}}-Gorenstein algebras and syzygy modules},
 fjournal = {Journal of Pure and Applied Algebra},
 journal = {J. Pure Appl. Algebra},
 issn = {0022-4049},
 volume = {92},
 number = {1},
 pages = {1--27},
 year = {1994},
 doi = {10.1016/0022-4049(94)90044-2},
 zbMATH = {540918},
 Zbl = {0803.16016}
}

@article{Huang1999,
 author = {Huang, Zhaoyong},
 title = {Extension closure of {{\(k\)}}-torsionfree modules},
 fjournal = {Communications in Algebra},
 journal = {Commun. Algebra},
 issn = {0092-7872},
 volume = {27},
 number = {3},
 pages = {1457--1464},
 year = {1999},
 doi = {10.1080/00927879908826506},
 zbMATH = {1270342},
 Zbl = {0928.16010}
}

@article{Huang-science,
 author = {Huang, Zhaoyong},
 title = {{{\(\mathbb{W}^t\)}}-approximation representations over quasi {{\(k\)}}-{Gorenstein} algebras},
 fjournal = {Science in China. Series A},
 journal = {Sci. China, Ser. A},
 issn = {1006-9283},
 volume = {42},
 number = {9},
 pages = {945--956},
 year = {1999},
 doi = {10.1007/BF02880386},
 zbMATH = {1439324},
 Zbl = {0966.16004}
}

@article{GHZ,
 author = {Gu, Weili and Huang, Zhaoyong and Zhao, Tiwei},
 title = {Homological invariants under {Frobenius} extensions},
 fjournal = {Colloquium Mathematicum},
 journal = {Colloq. Math.},
 volume = {178},
 pages = {77--95},
 year = {2025}
}

@article{Eisenbud1970,
 author = {Eisenbud, David},
 title = {Subrings of {Artinian} and {Noetherian} rings},
 fjournal = {Mathematische Annalen},
 journal = {Math. Ann.},
 issn = {0025-5831},
 volume = {185},
 pages = {247--249},
 year = {1970},
 doi = {10.1007/BF01350264},
 url = {https://eudml.org/doc/161955},
 zbMATH = {3329094},
 Zbl = {0207.04702}
}

@book{FGR,
 author = {Fossum, Robert M. and Griffith, Phillip A. and Reiten, Idun},
 title = {Trivial extensions of {Abelian} categories. {Homological} algebra of trivial extensions of {Abelian} categories with applications to ring theory},
 fseries = {Lecture Notes in Mathematics},
 series = {Lect. Notes Math.},
 issn = {0075-8434},
 volume = {456},
 year = {1975},
 publisher = {Springer, Cham},
 zbMATH = {3473879},
 Zbl = {0303.18006}
}

@article{Zhao2019,
 author = {Zhao, Zhibing},
 title = {Gorenstein homological invariant properties under {Frobenius} extensions},
 fjournal = {Science China. Mathematics},
 journal = {Sci. China, Math.},
 issn = {1674-7283},
 volume = {62},
 number = {12},
 pages = {2487--2496},
 year = {2019},
 doi = {10.1007/s11425-018-9432-2},
 zbMATH = {7137546},
 Zbl = {1468.16013}
}

@article{Ren-Corrigendum,
 author = {Ren, Wei},
 title = {Corrigendum to: ``{Gorenstein} projective and injective dimensions over {Frobenius} extensions''},
 fjournal = {Communications in Algebra},
 journal = {Commun. Algebra},
 issn = {0092-7872},
 volume = {48},
 number = {2},
 pages = {915--916},
 year = {2020},
 doi = {10.1080/00927872.2019.1654496},
 zbMATH = {7187974},
 Zbl = {1439.13024}
}

@misc{DKLO,
 author = {Dey, Souvik and Kimura, Kaito and Liu, Jian and Otake, Yuya},
 title = {On local rings of finite syzygy representation type},
 year = {},
 howpublished = {{arXiv}:2507.17097 (2025)},
 url = {https://arxiv.org/abs/2507.17097},
 arXiv = {arXiv:2507.17097}
}

@article{Ren-SCM,
 author = {Ren, Wei},
 title = {Gorenstein projective modules and {Frobenius} extensions},
 fjournal = {Science China. Mathematics},
 journal = {Sci. China, Math.},
 issn = {1674-7283},
 volume = {61},
 number = {7},
 pages = {1175--1186},
 year = {2018},
 doi = {10.1007/s11425-017-9138-y},
 zbMATH = {6916561},
 Zbl = {1401.16009}
}

@book{LW-book,
 author = {Leuschke, Graham J. and Wiegand, Roger},
 title = {Cohen-{Macaulay} representations},
 fseries = {Mathematical Surveys and Monographs},
 series = {Math. Surv. Monogr.},
 issn = {0076-5376},
 volume = {181},
 isbn = {978-0-8218-7581-0},
 year = {2012},
 publisher = {Providence, RI: American Mathematical Society (AMS)},
 zbMATH = {6039453},
 Zbl = {1252.13001}
}

@book{Lam,
 author = {Lam, Tsit-Yuen},
 title = {A first course in noncommutative rings},
 fseries = {Graduate Texts in Mathematics},
 series = {Grad. Texts Math.},
 issn = {0072-5285},
 volume = {131},
 isbn = {0-387-97523-3},
 year = {1991},
 publisher = {New York etc.: Springer-Verlag},
 zbMATH = {49937},
 Zbl = {0728.16001}
}

@book{Lam1999,
 author = {Lam, Tsit-Yuen},
 title = {Lectures on modules and rings},
 fseries = {Graduate Texts in Mathematics},
 series = {Grad. Texts Math.},
 issn = {0072-5285},
 volume = {189},
 isbn = {0-387-98428-3},
 year = {1999},
 publisher = {New York, NY: Springer},
 zbMATH = {1236965},
 Zbl = {0911.16001}
}

@book{ARS,
 author = {Auslander, Maurice and Reiten, Idun and Smal{\o}, Sverre O.},
 title = {Representation theory of {Artin} algebras},
 fseries = {Cambridge Studies in Advanced Mathematics},
 series = {Camb. Stud. Adv. Math.},
 volume = {36},
 isbn = {0-521-41134-3},
 year = {1995},
 publisher = {Cambridge: Cambridge University Press},
 zbMATH = {707210},
 Zbl = {0834.16001}
}

@article{Chen-Ren,
 author = {Chen, Xiao-Wu and Ren, Wei},
 title = {Frobenius functors and {Gorenstein} homological properties},
 fjournal = {Journal of Algebra},
 journal = {J. Algebra},
 issn = {0021-8693},
 volume = {610},
 pages = {18--37},
 year = {2022},
 doi = {10.1016/j.jalgebra.2022.06.030},
 zbMATH = {7588350},
 Zbl = {1503.18008}
}

@article{MTT,
 author = {Matsui, Hiroki and Takahashi, Ryo and Tsuchiya, Yoshinao},
 title = {When are {{\(n\)}}-syzygy modules {{\(n\)}}-torsionfree?},
 fjournal = {Archiv der Mathematik},
 journal = {Arch. Math.},
 issn = {0003-889X},
 volume = {108},
 number = {4},
 pages = {351--355},
 year = {2017},
 doi = {10.1007/s00013-017-1020-9},
 zbMATH = {6714146},
 Zbl = {1453.13044}
}

@Book{BH,
 Author = {Bruns, Winfried and Herzog, J{\"u}rgen},
 Title = {Cohen-{Macaulay} rings},
 Edition = {Rev. ed.},
 FSeries = {Cambridge Studies in Advanced Mathematics},
 Series = {Camb. Stud. Adv. Math.},
 Volume = {39},
 ISBN = {0-521-56674-6},
 Year = {1998},
 Publisher = {Cambridge: Cambridge University Press},
 zbMATH = {1194481},
 Zbl = {0909.13005}
}

@article{Liu-Ren,
title = {Ascent and descent of Gorenstein homological properties},
author={Liu, Jian and Ren, Wei},
fjournal = {Journal of Pure and Applied Algebra},
 journal = {J. Pure Appl. Algebra},
volume = {229},
pages = {108042},
year = {2025},
issn = {0022-4049},
doi = {https://doi.org/10.1016/j.jpaa.2025.108042},
url ={https://www.sciencedirect.com/science/article/pii/S0022404925001811}
}

@article{Holm,
 author = {Holm, Henrik},
 title = {Gorenstein homological dimensions},
 fjournal = {Journal of Pure and Applied Algebra},
 journal = {J. Pure Appl. Algebra},
 issn = {0022-4049},
 volume = {189},
 number = {1-3},
 pages = {167--193},
 year = {2004},
 doi = {10.1016/j.jpaa.2003.11.007},
 url = {curis.ku.dk/ws/files/41927598/GorensteinHomologicalDimensions.pdf},
 zbMATH = {2078857},
 Zbl = {1050.16003}
}

@article{Christensen2001,
 author = {Christensen, Lars Winther},
 title = {Semi-dualizing complexes and their {Auslander} categories},
 fjournal = {Transactions of the American Mathematical Society},
 journal = {Trans. Am. Math. Soc.},
 issn = {0002-9947},
 volume = {353},
 number = {5},
 pages = {1839--1883},
 year = {2001},
 doi = {10.1090/S0002-9947-01-02627-7},
 zbMATH = {1566245},
 Zbl = {0969.13006}
}

@article{Zhao-Sun,
 author = {Zhao, Guoqiang and Sun, Juxiang},
 title = {A note on {Gorenstein} transposes},
 fjournal = {Journal of Algebra and its Applications},
 journal = {J. Algebra Appl.},
 issn = {0219-4988},
 volume = {15},
 number = {10},
 pages = {8},
 note = {Id/No 1650180},
 year = {2016},
 doi = {10.1142/S0219498816501802},
 zbMATH = {6667895},
 Zbl = {1375.16006}
}

@article{Goto,
 author = {Goto, Shiro},
 title = {Vanishing of {{\(Ext^ i_ R(M,R)\)}}},
 fjournal = {Journal of Mathematics of Kyoto University},
 journal = {J. Math. Kyoto Univ.},
 issn = {0023-608X},
 volume = {22},
 pages = {481--484},
 year = {1982},
 doi = {10.1215/kjm/1250521731},
 zbMATH = {3836187},
 Zbl = {0527.13012}
}

@book{Auslander-Bridger,
 author = {Auslander, Maurice and Bridger, Mark},
 title = {Stable module theory},
 fseries = {Memoirs of the American Mathematical Society},
 series = {Mem. Am. Math. Soc.},
 issn = {0065-9266},
 volume = {94},
 isbn = {978-0-8218-1294-5; 978-1-4704-0044-6},
 year = {1969},
 publisher = {Providence, RI: American Mathematical Society (AMS)},
 doi = {10.1090/memo/0094},
 zbMATH = {3324867},
 Zbl = {0204.36402}
}

@article{Iyama2007,
 author = {Iyama, Osamu},
 title = {Higher-dimensional {Auslander}-{Reiten} theory on maximal orthogonal subcategories},
 fjournal = {Advances in Mathematics},
 journal = {Adv. Math.},
 issn = {0001-8708},
 volume = {210},
 number = {1},
 pages = {22--50},
 year = {2007},
 doi = {10.1016/j.aim.2006.06.002},
 zbMATH = {5132549},
 Zbl = {1115.16005}
}

@article{Chen,
 author = {Chen, Xiao-Wu},
 title = {Totally reflexive extensions and modules},
 fjournal = {Journal of Algebra},
 journal = {J. Algebra},
 issn = {0021-8693},
 volume = {379},
 pages = {322--332},
 year = {2013},
 doi = {10.1016/j.jalgebra.2013.01.014},
 url = {www.sciencedirect.com/science/article/pii/S0021869313000392},
 zbMATH = {6221937},
 Zbl = {1282.16020}
}

@article{Spaltenstein,
 author = {Spaltenstein, Nicolas},
 title = {Resolutions of unbounded complexes},
 fjournal = {Compositio Mathematica},
 journal = {Compos. Math.},
 issn = {0010-437X},
 volume = {65},
 number = {2},
 pages = {121--154},
 year = {1988},
 url = {https://eudml.org/doc/89885},
 zbMATH = {4036043},
 Zbl = {0636.18006}
}

@book{Kadison,
 author = {Kadison, Lars},
 title = {New examples of {Frobenius} extensions},
 fseries = {University Lecture Series},
 series = {Univ. Lect. Ser.},
 issn = {1047-3998},
 volume = {14},
 isbn = {0-8218-1962-3},
 year = {1999},
 publisher = {Providence, RI: American Mathematical Society},
 zbMATH = {1313912},
 Zbl = {0929.16036}
}

@misc{Kas,
 author = {Kasch, Friedrich},
 title = {Projektive {Frobenius}-{Erweiterungen}},
 year = {1961},
 howpublished = {Sitzungsber. {Heidelberger} {Akad}. {Wiss}., {Math}.-{Naturw}. {Kl}. 1960/61, 89-109 (1961).},
 zbMATH = {3170382},
 Zbl = {0104.26201}
}

@article{Zaks,
 author = {Zaks, Abraham},
 title = {Injective dimension of semi-primary rings},
 fjournal = {Journal of Algebra},
 journal = {J. Algebra},
 issn = {0021-8693},
 volume = {13},
 pages = {73--86},
 year = {1969},
 doi = {10.1016/0021-8693(69)90007-6},
 zbMATH = {3342993},
 Zbl = {0216.07001}
}

@book{Buchweitz,
 author = {Buchweitz, Ragnar-Olaf},
 title = {Maximal {Cohen}-{Macaulay} modules and {Tate} cohomology. {With} appendices by {Luchezar} {L}. {Avramov}, {Benjamin} {Briggs}, {Srikanth} {B}. {Iyengar} and {Janina} {C}. {Letz}},
 fseries = {Mathematical Surveys and Monographs},
 series = {Math. Surv. Monogr.},
 issn = {0076-5376},
 volume = {262},
 isbn = {978-1-4704-5340-4; 978-1-4704-6792-0},
 year = {2021},
 publisher = {Providence, RI: American Mathematical Society (AMS)},
 doi = {10.1090/surv/262},
 zbMATH = {7498869},
 Zbl = {1505.13002}
}

@book{EJ,
 author = {Enochs, Edgar E. and Jenda, Overtoun M. G.},
 title = {Relative homological algebra. {Vol}. 2},
 edition = {2nd revised ed.},
 fseries = {De Gruyter Expositions in Mathematics},
 series = {De Gruyter Expo. Math.},
 issn = {0938-6572},
 volume = {54},
 isbn = {978-3-11-021522-9; 978-3-11-021523-6},
 year = {2011},
 publisher = {Berlin: Walter de Gruyter},
 zbMATH = {5818169},
 Zbl = {1238.13002}
}

@book{Christensen-book,
 author = {Christensen, Lars Winther},
 title = {Gorenstein dimensions},
 fseries = {Lecture Notes in Mathematics},
 series = {Lect. Notes Math.},
 issn = {0075-8434},
 volume = {1747},
 isbn = {3-540-41132-1},
 year = {2000},
 publisher = {Berlin: Springer},
 doi = {10.1007/BFb0103980},
 zbMATH = {1544059},
 Zbl = {0965.13010}
}

@article{DT,
 author = {Dey, Souvik and Takahashi, Ryo},
 title = {On the subcategories of {{\(n\)}}-torsionfree modules and related modules},
 fjournal = {Collectanea Mathematica},
 journal = {Collect. Math.},
 issn = {0010-0757},
 volume = {74},
 number = {1},
 pages = {113--132},
 year = {2023},
 doi = {10.1007/s13348-021-00338-1},
 zbMATH = {7640070},
 Zbl = {1524.13050}
}

@book{Weibel,
 author = {Weibel, Charles A.},
 title = {An introduction to homological algebra},
 fseries = {Cambridge Studies in Advanced Mathematics},
 series = {Camb. Stud. Adv. Math.},
 volume = {38},
 isbn = {0-521-43500-5},
 year = {1994},
 publisher = {Cambridge: Cambridge University Press},
 zbMATH = {595200},
 Zbl = {0797.18001}
}

@article{Takahashi2013,
 author = {Takahashi, Ryo},
 title = {Classifying resolving subcategories over a {Cohen}-{Macaulay} local ring},
 fjournal = {Mathematische Zeitschrift},
 journal = {Math. Z.},
 issn = {0025-5874},
 volume = {273},
 number = {1-2},
 pages = {569--587},
 year = {2013},
 doi = {10.1007/s00209-012-1020-1},
 zbMATH = {6134692},
 Zbl = {1267.13024}
}

@article{IZ,
 author = {Iyama, Osamu and Zhang, Xiaojin},
 title = {Tilting modules over {Auslander}-{Gorenstein} algebras},
 fjournal = {Pacific Journal of Mathematics},
 journal = {Pac. J. Math.},
 issn = {1945-5844},
 volume = {298},
 number = {2},
 pages = {399--416},
 year = {2019},
 doi = {10.2140/pjm.2019.298.399},
 zbMATH = {7050834},
 Zbl = {1470.16026}
}
\end{document}